\documentclass[preprint]{imsart}

\RequirePackage{amsthm,amsmath,amsfonts,amssymb}
\RequirePackage[numbers]{natbib}
\RequirePackage{mathtools,dsfont}
\RequirePackage[colorlinks,citecolor=blue,urlcolor=blue]{hyperref}
\RequirePackage{graphicx}

\startlocaldefs
\numberwithin{equation}{section}
\theoremstyle{plain}
\newtheorem{theorem}{Theorem}[section]

\newtheorem{lemma}[theorem]{Lemma}
\newtheorem{proposition}[theorem]{Proposition}
\newtheorem{cor}[theorem]{Corollary}
\theoremstyle{remark}
\newtheorem{remark}[theorem]{Remark}
\newtheorem{assumption}[theorem]{Assumption}
\DeclarePairedDelimiter{\floor}{\lfloor}{\rfloor}

\newcommand{\N}{\mathbb{N}}
\newcommand{\R}{\mathbb{R}}

\newcommand{\Z}{\mathbb{Z}}

\newcommand{\p}{\mathbb{P}}

\newcommand{\cB}{\mathcal{B}}
\newcommand{\cC}{\mathcal{C}}
\newcommand{\cD}{\mathcal{D}}

\newcommand{\cF}{\mathcal{F}}

\newcommand{\cL}{\mathcal{L}}

\newcommand{\cU}{\mathcal{U}}

\newcommand{\cV}{\mathcal{V}}
\newcommand{\cW}{\mathcal{W}}
\newcommand{\cY}{\mathcal{Y}}
\newcommand{\cX}{\mathcal{X}}

\newcommand{\fD}{\mathfrak{D}}

\newcommand{\E}{\mathbb{E}}
\newcommand{\Cov}{\mathrm{Cov}}
\newcommand{\V}{\operatorname{Var}} 
\renewcommand{\phi}{\varphi}
\renewcommand{\epsilon}{\varepsilon}
\newcommand{\diff}{\mathrm{d}}

\newcommand{\diam}{\operatorname{diam}}
\newcommand{\ul}{\underline}
\newcommand{\ol}{\overline}
\newcommand{\wt}{\widetilde}
\newcommand{\wh}{\widehat}
\newcommand{\1}[1]{\,\mathds{1}\! \left\{ #1 \right\} }

\newcommand{\dN}{\mathrm{N}}

\newcommand{\lf}{\lfloor}
\newcommand{\rf}{\rfloor}

\newcommand{\pd}{\operatorname{PD}}
\newcommand{\pers}{\operatorname{Pers}}

\DeclareMathOperator*{\argmin}{arg\,min}

\endlocaldefs

\allowdisplaybreaks

\begin{document}

\begin{frontmatter}
\title{Two-sample tests for relevant differences in persistence diagrams}
\runtitle{Two-sample tests in persistence diagrams}
\begin{aug}

\author{\fnms{Johannes} \snm{Krebs}  
\ead[label=e1]{johannes.krebs@ku.de}}
\address{KU Eichstätt, Ostenstraße 28, 85072 Eichstätt, Germany.
}

\author{\fnms{Daniel} \snm{Rademacher}  
\ead[label=e2]{daniel.rademacher@uni-heidelberg.de}}
\address{Heidelberg University, Im Neuenheimer Feld 205, 69120 Heidelberg, Germany.
}

\runauthor{Krebs and Rademacher}
\affiliation{KU Eichstätt and Heidelberg University}
\end{aug}

\begin{abstract}
We study two-sample tests for relevant differences in persistence diagrams obtained from $L^p$-$m$-approximable data $(\cX_t)_t$ and $(\cY_t)_t$. To this end, we compare  variance estimates w.r.t.\ the Wasserstein metrics on the space of persistence diagrams. In detail, we consider two test procedures. The first compares the Fr{\'e}chet variances of the two samples based on estimators for the Fr{\'e}chet mean of the observed persistence diagrams $\pd(\cX_i)$ ($1\le i\le m$), resp., $\pd(\cY_j)$ ($1\le j\le n$) of a given feature dimension. We use classical functional central limit theorems to establish consistency of the testing procedure. The second procedure relies on a comparison of the so-called independent copy variances of the respective samples. Technically, this leads to functional central limit theorems for U-statistics built on $L^p$-$m$-approximable sample data.
\end{abstract}

\begin{keyword}[class=MSC2020]
\kwd[Primary ]{62R40}
\kwd{62G10}
\kwd[; secondary ]{62R20}
\kwd{60F17} 
\end{keyword}

\begin{keyword}
\kwd{Functional central limit theorems} \kwd{$L^p$-$m$-approximable data} \kwd{persistence diagrams} \kwd{relevant differences} 
\kwd{self-normalization} \kwd{topological data analysis} \kwd{two sample tests} \kwd{U-statistics}
\end{keyword}
%

%
%
\end{frontmatter}


We study two-sample tests for relevant differences in the (Fr{\'e}chet) variances of observed persistence diagrams $\pd(\cX_i)$ ($1\le i\le m$), resp., $\pd(\cY_j)$ ($1\le j\le n$) which each feature a weak dependence relation. Statistical tests on relevant differences have enjoyed quite some popularity recently. In particular, in the domain of functional time series there are major contributions; we refer to Dette and Wied \cite{dette2016detecting}, Dette and Wu \cite{dette2019detecting} and Aue et al.\ \cite{aue2019two}. This testing problem has not received much attention in topological data analysis (TDA) so far.

Classical two-sample tests focus on hypotheses of the form
\begin{align*}
	&H_0: T_\cX=  T_\cY \text{ vs. } H_1: T_\cX \neq T_\cY,
\end{align*}
for suitable real-valued statistics $T_\cX = T\big( (\pd(\cX_i))_{i= 1}^m \big)$ and $T_\cY = T\big((\pd(\cY_j))_{j=1}^n\big)$. In contrast, we study relevant hypotheses of the form
\begin{align*}
	&H_0: | T_\cX - T_\cY  |^2 \le \Delta \text{ vs. } H_1: |T_\cX- T_\cY  |^2 > \Delta.
\end{align*}
The statistics $T_\cX$ and $T_\cY$ which we consider are variances estimates of a sequence of persistence diagrams $\pd(\cX_i)$ ($1\le i\le m$), resp., $\pd(\cY_j)$ ($1\le j\le n$). The underlying data generating processes $(\cX_t)_t$ and $(\cY_t)_t$ are $L^p$-$m$-approximable processes with values in a suitable subset $\fD$ of $\R^d$. We assume that the diagrams are built with the \v Cech filtration function which is a classical choice in TDA. We refer to the book of Boissonnat et al.\ \cite{boissonnat2018geometric} and the introductory article of Chazal and Michel \cite{chazal2021introduction} for a background on the subject, in particular, the notion of a filtration function in TDA and the construction of persistence diagrams from simplicial complexes.

Early celebrated contributions which highlighted the potential of TDA are Edelsbrunner et al.\ \cite{edelsbrunner2000topological}, Zomorodian and Carlsson \cite{zomorodian2005computing} and Carlsson et al.\ \cite{carlsson2009topology}. A vast share of recent contributions in TDA focus on computational and practical aspects of TDA, nevertheless the methods have received some attention in the area of (statistical) time series as well, we refer to Seversky et al.\ \cite{seversky2016time}, Umeda \cite{umeda2017time}. 

However, despite the widespread dissemination across various disciplines, for large parts of TDA, the development of mathematically sound statistical tests is still in its infancy.
Our contribution to this matter is as follows. We study two two-sample tests. The first is based on Fr{\'e}chet means and Fr{\'e}chet variances. Properties of Fr{\'e}chet variances of a sequence iid data objects have been studied by Dubey and Müller \cite{dubey2019frechet, dubey2020frechet}, e.g., in the context of structural breaks.
The second approach is via so-called independent copy (inco)-variances which are a special case of U-statistics (Hoeffding \cite{hoeffding1948StatisticsWithAsympNormalDistribution}). Dehling and Wendler \cite{dehling2010CLTandBootstrapForUStatistics}, Leucht \cite{leucht2012DegenerateUandVStatistics} study U-statistics the context of dependent data. 
Our contribution is as follows: We develop in both settings self-normalized tests for relevant differences in the variances of persistence diagrams which are generated by weakly depedent data. Moreover, we develop a function central limit theorem for general U-statistics of weakly dependent data because these underlying structure for the inco-variance approach.

We conclude with some notation.
We will assume that $\R^d$ is equipped with some $l^r$-norm, i.e., $\|x\| = \|x\|_r = (\sum_{i=1}^d |x_i|^r)^{1/r}$ for some $r\geq 1$.
	Also for $x,y\in\R^d$ we may write $(x,y)$ to denote the $\R^{2d}$-valued vector $(x_1,\ldots,x_d,y_1,\ldots,y_d)$ and -- unless otherwise stated -- all vector	quantities are always viewed as column-vectors.
The $d$-dimensional Lebesgue measure of a Borel set $A\subseteq\R^d$ is denoted by $|A|$.
The closed $a$-ball at a point $x\in\R^d$ is the set $B(x,a)=\{y: \|x-y\|_2\le a \}$.

The rest of the manuscript is organized as follows. 
Section~\ref{Section_Model} contains a short introduction to the main objects from TDA as well a detailed description of the framework for the two sample tests. 
The Subsection~\ref{Subsection_TestFrechetMean} contains the results for two sample tests based on Fr{\'e}chet means, while the Subsection~\ref{Subsection_IncoVariances} 
treats the case for the inco-variances. 

The mathematical background on U-statistics is treated in Section~\ref{Section_Background}. All mathematical details are deferred to Section~\ref{Section_MathematicalDetails}.

\section{Tests for persistence diagrams}\label{Section_Model}
Throughout the article, we assume that the underlying simplicial complexes are all built with the \v Cech filtration function. That is, given a finite $P\subseteq\R^d$, a simplex $\{x_0,\ldots,x_k\} \subseteq P$ is an element of the \v Cech complex at parameter $a>0$ if and only if $\bigcap_{i=0}^k B(x_i,a) \not = \emptyset$. We write $\cC(P,a)$ for this simplicial complex.

Let $\mathcal{D}$ denote the set of all persistence diagrams and consider the $L^p$-Wasserstein metrics 
\begin{align*}
	W_r(U,V) &= \inf_{\gamma } \left( \sum_{u\in U} \|u-\gamma(u)\|_\infty^r \right)^{1/r}, \quad r\in [1,\infty),\\
	W_\infty(U,V) &= \inf_{\gamma}  \sup_{u\in U} \|u-\gamma(u)\|_\infty.
\end{align*}
$W_\infty$ is also known as the bottleneck distance $W_B$.
Then, it is true that $W_\infty(U,V) \le W_q(U,V) \le W_r(U,V) \le \max\{ \# U, \# V \}^{1/r-1/q} W_q(U,V)$ for $1\le r \le q \le \infty$ by the usual norm inequalities in the 
Euclidean space.

Define for $r\in [1,\infty]$ the sets
$$
	\mathcal{D}_r := \{ U\in\cD \mid W_r(U, U_{\emptyset}) < \infty\},
$$
where $U_{\emptyset}=\{(x,x): x\ge 0\}$ is a diagram with just the diagonal. 
Then $(\mathcal{D}_r, W_r)$ is a complete metric space and a set $S\subset \mathcal{D}_r$ is totally bounded if and only if it is bounded, off-diagonally birth-death bounded and 
uniform, see Mileyko et al. \cite{mileyko2011probability}[Theorem 6 and 21] for more details, in particular the notion of `uniform'. Moreover, it is known that $(\mathcal{D}_2, W_2)$ is a non-negatively curved Alexandrov space for $p=2$, see Turner et al. \cite{turner2014frechet}[Theorem 2.5].

We equip the metric space $(\cD_r,W_r)$ with its Borel-$\sigma$-field $\cB(\cD_r)$. To any a probability measure $\rho$ on $(\cD_r, \cB(\cD_r))$, there corresponds a so-called Fr\'echet function 
\begin{align}\label{Eq:FrechetFunction1}
 F_{\rho} \colon \cD_r \to \R, \quad V \mapsto \int_{\cD_r} W_r^2(U,V) \,\rho(\diff U). 
 \end{align}
The \emph{Fr\'echet variance} $\V(\rho) \in \R$ and the \emph{Fr\'echet mean(s)} $\mu(\rho) \subseteq \cD_r$ of $\rho$ are then defined as the solutions of the optimization problem
\begin{align}\label{Eq:FrechetMeanAndVariance}
	\V(\rho) = \inf_{V\in \cD_r} F_\rho(V) , \quad \mu(\rho) = \argmin_{V\in \cD_r}F_\rho(V).
\end{align}
If the probability measure $\rho$ is induced by a $\cD_r$-valued random element $X$, the corresponding Fr\'echet function can simply be expressed via an expectation 
\begin{align}\label{Eq:FrechetFunction2}
	F_X\colon \cD_r \to \R,\quad V \mapsto \E[W_r^2(X,V)].
\end{align}
We may write $\V(X)$ respectively $\mu(X)$ for the solutions of \eqref{Eq:FrechetMeanAndVariance} and refer to these quantities as Fr\'echet variance respectively Fr\'echet mean(s) of $X$.
The notational similarity to the Euclidean mean and variance is of course intentional, since the Fr\'echet mean and variance poses a coherent generalization to metric spaces. Especially, the
Fr\'echet variance provides a natural measure for dispersion. 

Under mild assumptions on $\rho$ (e.g., bounded support), respectively $X$, existence of a Fr\'echet mean is guaranteed, we refer to Mileyko et al. \cite{mileyko2011probability}
[Theorem 21 and Lemma 27] for a detailed study on Fr\'echet means of persistence diagrams.

In particular, we have for the empirical measure $\rho_m = \frac{1}{m} \sum_{i=1}^{m} \delta_{\pd(\cX_i)}$ based on some $\R^d$-valued point cloud process $\cX = (\cX_t)_t$
\begin{align*}
	F_{\rho_m}(V) = \frac{1}{m} \sum_{i=1}^m W_r^2(\pd(\cX_i),V),\quad V\in \cD_r.
\end{align*}
Assuming stationarity of an underlying $\R^d$-valued point cloud process $\cX = (\cX_t)_t$ implies that the sequence $(\pd(\cX_t))_t$ of corresponding persistence diagrams of a given feature dimension $k\in\{0,\ldots,d-1\}$ is also a stationary process in the metric space $(\mathcal{D}_r, W_r)$, that is $\pd(\cX_t) \sim \rho_{\cX}$ for some probability measure $\rho_{\cX}$. Given a sample $\cX_1,\ldots,\cX_m$ stemming from the process $\cX$, this suggests 
\begin{align*}
	&\arg\min_{V\in \cD_r} \frac{1}{m} \sum_{i=1}^m W_r^2(\pd(\cX_i),V) \text{ and } 
	\wh{\V}(\cX) = \frac{1}{m} \sum_{i=1}^m W_r^2(\pd(\cX_i),\wh{\mu}_\cX) 
\end{align*}
as suitable estimates for the Fr\'echet mean $\mu(\cX) = \mu_{\rho_{\cX}}$ and the Fr\'echet variance $\V(\cX) = \V_{\rho_{\cX}}$ of the 
persistence diagram $\pd(\cX_0)$.

An alternative measure for the dispersion of persistence diagrams corresponding to a (stationary) point cloud process $\cX = (\cX_t)_t$ is given by the so-called 
\emph{independent copy (inco)-variance}
\begin{align}\label{Eq:InCoVariancePD}
	\sigma^2(\cX) := \frac{1}{2} \E[ W_r^2(\pd(\cX_0), \pd(\cX')) ],
\end{align}
where $\cX'$ is independent of $\cX$ and has the same distribution as $\cX_0$. The motivation for \eqref{Eq:InCoVariancePD} stems from the simple 
variance identity $\V(X) = \E[(X-X')^2]/2$, where $X$ is an $\R$-valued random variable with finite second moment and $X'$ is an independent copy. 
However, contrary to the Euclidean case, the Fr\'echet variance and inco-variance do not coincide in general, but a simple calculation shows 
\begin{align*}
	 \V(\cX) - \E[ W_r(\pd(\cX_0), \pd(\cX')) ]^2 \leq \sigma^2(\cX) \leq 2 \V(\cX).
\end{align*} 
Since determining $\V(\cX)$ requires a computation or estimation 
of $\mu(\cX)$, which in turn leads to a complex minimization problem with possibly non-unique solutions, the much simpler form of the inco-variance makes 
it a convenient candidate as a dispersion measure -- not only from a computational point of view. 
As an estimator for $\sigma^2(\cX)$ we consider the U-statistic 
\begin{align}\label{Eq:InCoVariancePDEstimator}
	\wh\sigma^2(\cX) := \frac{1}{m(m-1)} \sum_{i,j=1\atop i\not= j}^m \frac{1}{2} W_r^2(\pd(\cX_i), \pd(\cX_j)).
\end{align}
In the following, let $\fD \subseteq \R^d$ be a Borel set. Usually, we choose $\fD$ bounded and convex.
Assume that we are given two samples $\cX_1,\ldots,\cX_m$ and $\cY_1,\ldots,\cY_n$ stemming from independent stationary, $\fD$-valued point cloud processes 
$\cX=(\cX_t)_t$ and $\cY=(\cY_t)_t$ with common distribution $\p_\cX$, resp., $\p_\cY$. We are interested in testing the second order similarity 
of corresponding persistence diagrams $\pd(\cX_0)$ and $\pd(\cY_0)$ for a given feature dimension $k\in\{0,\ldots,d-1\}$
by comparing their Fr\'echet variances $\V(\cX)$ and $\V(\cY)$ or alternatively their independent copy variances $\sigma^2(\cX)$ and $\sigma^2(\cY)$. 

Since it is unlikely that these quantities are exactly identical it seems more reasonable to test whether their difference is substantial. This can be 
formulated as testing problem for relevant differences
\begin{align}\label{Testproblem:FrechetVariance}
	H_0\colon (\V_\cX - \V_\cY)^2 \leq \Delta \quad \text{vs.} \quad H_1\colon (\V_\cX - \V_\cY)^2 > \Delta,
\end{align}
respectively, 
\begin{align}\label{Testproblem:IncoVariance}
	H_0\colon (\sigma_{\cX}^2 - \sigma_{\cY}^2)^2 \leq \Delta \quad \text{vs.} \quad H_1\colon (\sigma_{\cX}^2 - \sigma_{\cY}^2)^2 > \Delta,
\end{align}  
where $\Delta \geq 0$ is a threshold parameter (set externally).

\subsection{A test statistic for relevant differences of Fr\'echet variances}\label{Subsection_TestFrechetMean}
We begin with the problem \eqref{Testproblem:FrechetVariance} and test whether the difference $D=\V(\cX) - \V(\cY)$ of the 
Fr\'echet variances of the persistence diagrams $\pd(\cX_0)$ and $\pd(\cY_0)$ is substantial. 
To this end, we need to consider the partial-sum process version of the variance estimate, that is
\begin{align*}
	\wh{\V}(\cX)(s) = \frac{1}{\lf ms \rf} \sum_{i=1}^{\lf ms \rf} W_r^2(\pd(\cX_i),\wh{\mu}(\cX)), \qquad s \in[0,1],\\
	\wh{\V}(\cY)(s) = \frac{1}{\lf ns \rf} \sum_{i=1}^{\lf ns \rf} W_r^2(\pd(\cY_i),\wh{\mu}(\cY)),\qquad s \in[0,1].
\end{align*}
We work with
\begin{assumption}[\bf{r}]\label{A:FrechetVariances} Let $r> 4$.
\begin{itemize}
	\item [(1)] {\it Moment condition.} $\E[ W_r^4( \pd(\cX_0), \mu_\cX ) ] < \infty$, $\E[ W_r^4( \pd(\cY_0), \mu_\cY ) ] < \infty$.
	\item [(2)] {\it $L^r$-$m$-approximability.} The stationary point cloud processes $(\cX_t)_t$ and $(\cY_t)_t$ each allow for a Bernoulli shift representation, i.e.,
	$\cX_t = f(\epsilon_t,\epsilon_{t-1},\ldots)$ and $\cY_t = g(\eta_t,\eta_{t-1},\ldots)$. Here $f\colon S^{\infty} \to \fD$, resp.,
	$g\colon E^{\infty} \to \fD$ are measurable functions and $(\epsilon_t)_t$, resp., $(\eta_t)_t$ are i.i.d. random elements with values in
	some measurable space $S$, resp., $E$. Furthermore, let $(\cX_t^{(m)})_t$ and $(\cY_t^{(m)})_t$ denote the $m$-dependent approximation processes given by
	\begin{align*}
		\cX_t^{(m)} = f(\epsilon_t,\ldots,\epsilon_{t-m+1}, \epsilon_{t-m}^{(t)}, \epsilon_{t-m-1}^{(t)},\ldots ),\\
		\cY_t^{(m)} = g(\eta_t,\ldots,\eta_{t-m+1}, \eta_{t-m}^{(t)}, \eta_{t-m-1}^{(t)}, \ldots ),
	\end{align*}
	where, for each \emph{fixed} $t\in\Z$, the sequences $(\epsilon_j^{t})_j$, resp., $(\eta_j^{t})_j$ are independent copies of the original innovation
	processes $(\epsilon_t)_t$, resp., $(\eta_t)_t$. Then $(\cX_t)_t$ and $(\cY_t)_t$ are $L^r$-$m$-approximable w.r.t.\ the Hausdorff distance in the sense that 
	\begin{align*}
	\sum_{m\ge 1} \E[ d_H( \cX_0,\cX_0^{(m)})^{r} ]^{1/r} < \infty \quad\text{ resp., }\quad \sum_{m\ge 1} \E[ d_H( \cY_0,\cY_0^{(m)})^{r} ]^{1/r} < \infty.
	\end{align*}
	\item [(3)] {\it Finite diagrams.} The number of features in a diagram $\#\pd(\cX_0)$, resp., $\#\pd(\cY_0)$ satisfies 
	\begin{align*}
		\E[ \#\pd(\cX_0)^{4/(r-4)} ] < \infty, \quad\text{ resp., }\quad \E[ \#\pd(\cY_0)^{4/(r-4)} ] < \infty.
	\end{align*}
	
	\item [(4)] {\it Consistent estimators.} There are unique and consistent estimators $\wh\mu(\cX)$ and $\wh\mu(\cY)$ for the Fr\'echet means $\mu(\cX)$ and $\mu(\cY)$ 
	in the sense that
	\begin{align*}
	 W_r(\wh\mu(\cX),\mu(\cX)) = o_\p(1) \quad (m\to\infty) \quad \text{ resp., } \quad  W_r(\wh\mu(\cY),\mu(\cY)) = o_\p(1) \quad (n\to\infty). 
	\end{align*}
	Moreover, these estimators attain the following implicit rate
	\begin{align*}
		\max_{m^{1/2}\le k\le m}  \frac{1}{\sqrt{k}} \Big|\sum_{i=1}^k W_r^2( \pd(\cX_i),\wh\mu_\cX) - W_r^2(\pd(\cX_i),\mu_\cX) \Big| = o_\p(1) \quad (m\to\infty)
	\end{align*}
	and
	\begin{align*}
		\max_{n^{1/2}\le \ell\le n} \frac{1}{\sqrt{\ell}}  \Big|\sum_{j=1}^\ell W_r^2( \pd(\cY_j),\wh\mu_\cY) - W_r^2(\pd(\cY_j),\mu_\cY) \Big| = o_\p(1) \quad (n\to\infty).
	\end{align*}
\end{itemize}
\end{assumption}
\noindent Moreover, we will rely on the following difference process 
\begin{align*}
	\wh{D}_{m,n}(s) = s\,\big(\wh{\V}(\cX)(s) - \wh{\V}(\cY)(s)\big), \qquad s \in[0,1].
\end{align*}
In particular, we set $\wh{D}_{m,n} = \wh{D}_{m,n}(1) = \wh{\V}(\cX) - \wh{\V}(\cY)$.
Then, given  a tolerance $\Delta > 0$, the null hypothesis in \eqref{Testproblem:FrechetVariance} is rejected if $\wh{D}_{m,n}^2 > \Delta$. 
As common in two sample problems, the main difficulty 
here is to make the statistic $\wh{D}_{m,n}$ pivotal. To that end, we define the self-normalizer 
\begin{align*}
	\wh{V}_{m,n} &= \left\{ \int_0^1 \big(\wh{D}_{m,n}^2(s) - s^2\wh{D}_{m,n}^2 \big)^2 \,\nu(\diff s) \right\}^{1/2},
\end{align*}
where $\nu$ is some probability measure on $[0,1]$. Dette and Kokot \cite{dette2022detecting} and Aue et al. \cite{aue2019two} study similar normalizers. 
We have the following result:

\begin{theorem}\label{T:FrechetVarianceJointConvergence}
Let the regularity conditions of Assumption~\ref{A:FrechetVariances} (r) be satisfied for some $r>4$ and suppose there is a constant $\tau \in (0,1)$ such that 
$\lim_{m,n \to \infty}m/(m+n) = \tau$. Moreover, let $X^*_t = W_r^2(\pd(\cX_t), \wh\mu_\cX)$, resp., $Y^*_t = W_r^2(\pd(\cY_t), \wh\mu_\cY)$ for $t\in\Z$. Then
	\begin{align*}
	\sqrt{m+n}\Big( \wh{D}_{m,n}^2 - D^2,~\wh{V}_{m,n}\Big) \Rightarrow 
	\left( \xi B(1),~ \xi \left\{ \int_0^1 s^2\big(B(s) - s B(1)\big)^2 \,\nu(\diff s) \right\}^{1/2} \right),
	\end{align*}
where $B$ is a standard Brownian motion, $\xi = 2\sqrt{\frac{\Gamma_\cX}{\tau} + \frac{\Gamma_\cY}{1-\tau} }(\E[X^*_0]-\E[Y^*_0])$ and where 
$\Gamma_{\cX} = \sum_{k\in\Z} \Cov(X^*_0,X^*_k)$, resp., $\Gamma_{\cY} = \sum_{k\in\Z} \Cov(Y^*_0,Y^*_k)$. 
\end{theorem}

Clearly, the constant $\xi$ is infeasible in practice, however, the applied test statistic has a self normalizing property, so in the limit this constant will disappear. Now, Theorem \ref{T:FrechetVarianceJointConvergence} leads to the test-statistic
\begin{align}\label{E:FrechetVarTest}
	\wh{W}_{m,n} = \frac{\wh{D}_{m,n}^2 - \Delta}{\wh{V}_{m,n}}
\end{align}
and suggests to reject the hypothesis $H_0$ in \eqref{Testproblem:IncoVariance} if $\wh{W}_{m,n} > q_{1-\alpha}$, where $q_{1-\alpha}$ is 
the $(1-\alpha)$-quantile of
\begin{align}\label{E:LimitFrechetVarStatistic}
	W = \frac{B(1)}{\left\{ \int_0^1 s^2\big(B(s) - sB(1)\big)^2 \, \nu(\diff s) \right\}^{1/2}}.
\end{align}  
As an immediate consequence we have that the test \eqref{E:FrechetVarTest} is asymptotically consistent for the relevant hypotheses in \eqref{Testproblem:IncoVariance}:
\begin{theorem}\label{T:FrechetVarTest} Let the conditions of Theorem~\ref{T:FrechetVarianceJointConvergence} be satisfied. Then 	
	\begin{align}\label{E:AsymConsistencyFrechetVarTest}
		\lim_{m,n \to \infty} \p( \wh{W}_{m,n} > q_{1-\alpha} ) 
		= \begin{cases}
			1 &, D^2 > \Delta,\\
			\alpha &, D^2 = \Delta,\\
			0 &, D^2 < \Delta.
		\end{cases} 
	\end{align}
\end{theorem}
\begin{proof}[Proof of Theorem~\ref{T:FrechetVarTest}]
The proof is immediate. We have
\begin{align*}
	\p(\wh{W}_{m,n} > q_{1-\alpha}) 
	= \p\left( \frac{\wh{D}_{m,n}^2 - D^2}{\wh{V}_{m,n}} > q_{1-\alpha} + \frac{\Delta - D^2}{\wh{V}_{m,n}} \right)
\end{align*}
and Theorem~\ref{T:FrechetVarianceJointConvergence} implies  $(\wh{D}_{m,n}^2 - D^2)/\wh{V}_{m,n} \Rightarrow W$
as well as 
\begin{align*}
	\frac{\Delta - D^2}{\wh{V}_{m,n}} 
	\rightarrow
	\begin{cases}
		-\infty &, D^2 > \Delta,\\
		0 &, D^2 = \Delta,\\
		\infty &, D^2 < \Delta.
	\end{cases} 
\end{align*}
This proves the assertion.
\end{proof}

\subsection{A test statistic for relevant differences of inco-variances}\label{Subsection_IncoVariances}

We derive a test for the second problem \eqref{Testproblem:IncoVariance}, that is whether the difference $D=\sigma_{\cX}^2 - \sigma_{\cY}^2 $ of the 
independent copy variances is substantial. To this end, we consider the empirical counterpart $\wh{D}_{m,n} := \wh{D}_{m,n}(1,1)$, where
\begin{align}\label{E:IncoVarDifferenceProcess}
	\wh{D}_{m,n}(s,t) := \wh\sigma_{\cX}^2(s,t) - \wh\sigma_{\cY}^2(s,t), \qquad 0\leq s,t \leq 1,
\end{align}
is the difference of the two parameter inco-variance partial sum processes 
\begin{align}
	\begin{split}\label{E:IncoVarProcess}
	\wh\sigma_{\cX}^2(s,t)
	&= \frac{1}{m(m-1)} \sum_{i=1}^{\lf ms \rf} \sum_{j=1\atop j\not= i}^{\lf mt \rf} \frac{1}{2} W_r^2(\pd(\cX_i), \pd(\cX_j)),\\
	\wh\sigma_{\cY}^2(s,t)
	&= \frac{1}{n(n-1)} \sum_{i=1}^{\lf ns \rf} \sum_{j=1\atop j\not= i}^{\lf nt \rf} \frac{1}{2} W_r^2(\pd(\cY_i), \pd(\cY_j)).
	\end{split}
\end{align}
We intentionally reuse some of the symbols of the previous section here to keep the notation simple and since it is always clear from the context, which testing problem is considered. Also we put $\wh\sigma_{\cX}^2 = \wh\sigma_{\cX}^2(1,1)$ and  
$\wh\sigma_{\cY}^2 = \wh\sigma_{\cY}^2(1,1)$, so that $\wh{D}_{m,n} = \wh\sigma_{\cX}^2 - \wh\sigma_{\cY}^2$.
Moreover, in order to exploit the results of Section \ref{Section_Background}, we introduce the kernel function 
\begin{align}\label{E:KernelU-StatisticForIncoVar}
	h\colon \dN \times \dN \to \R, \qquad h(\cdot,\cdot) := 2^{-1} W_r^2( \pd(\cdot), \pd(\cdot) ),
\end{align}
so that \eqref{E:IncoVarProcess} takes the form of U-statistics with kernel $h$.

Then, given a tolerance $\Delta>0$, the null hypothesis is rejected if $\wh{D}_{m,n}^2 > \Delta$. Again, the main difficulty is
to make the test statistic $\wh{D}_{m,n}^2$ pivotal. To that end, let us introduce the following self-normalizer
\begin{align}\label{E:NormalizerIncoVarTest}
	\wh{V}_{m,n} := \left\{ \int_0^1 \int_0^1 \Big[ \wh{D}_{m,n}^2(s,t) - (st\wh{D}_{m,n})^2  \Big]^2 \, \nu(\diff s,\diff t) \right\}^{1/2},
\end{align}
where $\nu$ is a probability measure on $[0,1]^2$.

\begin{theorem}\label{Thrm:IncoVarianceJointConvergence}
Suppose $\lim_{n,m\to \infty} m/(m+n) = \tau \in (0,1)$. Either assume

 Case (a): $\# \cX_i, \# \cY_j \le C$ with probability 1 for all $i,j\in\Z$ for a certain $C\in\R_+$ and
 $$
 	\sum_{m \ge 1} m~ \E[ d_H( \cX_0, \cX_0^{(m)} )^{1-1/r} ] < \infty \text{ and } \sum_{m \ge 1} m ~ \E[ d_H( \cY_0, \cY_0^{(m)} )^{1-1/r} ] < \infty 
 $$
or assume

Case (b): Both $d<r$ as well as
$$
	\sum_{m \ge 1} m~ \E[ d_H( \cX_0, \cX_0^{(m)} )^{\rho} ] < \infty \text{ and } \sum_{m \ge 1} m ~ \E[ d_H( \cY_0, \cY_0^{(m)} )^{\rho} ] < \infty 
$$
for some $\rho\in (0,1-d/r)$. 

Then
\begin{align}
\begin{split}\label{E:IncoVarianceJointConvergence0}
	&\sqrt{m+n}( \wh{D}_{m,n}^2 - D^2, \wh{V}_{m,n} )\\
	\Rightarrow &\left( 2\xi B(1), \left\{ \xi^2 \int_0^1 \int_0^1 \Big[ st\big(tB(s) + sB(t) - 2stB(1)\big) \Big]^2  
	\, d\nu(s,t)\right\}^{1/2} \right),
\end{split}
\end{align}
where $\xi = 2\sqrt{\frac{\Gamma_\cX}{\tau} + \frac{\Gamma_\cY}{1-\tau} }(\sigma_{\cX}^2 - \sigma_{\cY}^2 )$ and where 
$\Gamma_\cX = \sum_k \Cov(h_\cX(\cX_0), h_\cX(\cX_k))$, resp., $\Gamma_\cY = \sum_k \Cov(h_\cY(\cY_0), h_\cY(\cY_k))$ for 
\begin{align*}
	h_\cX(\cdot) =& \E\Big[ \frac{1}{2} W_r^2( \pd(\cdot), \pd(\cX') )\Big], \qquad \cX' \sim \p_\cX,\\
	h_\cY(\cdot) =& \E\Big[ \frac{1}{2} W_r^2( \pd(\cdot), \pd(\cY') ) \Big], \qquad \cY' \sim \p_\cY.
\end{align*}
\end{theorem}
Theorem \ref{Thrm:IncoVarianceJointConvergence} leads to the test-statistic
\begin{align*}
	\wh{W}_{m,n} = \frac{\wh{D}_{m,n}^2 - \Delta}{\wh{V}_{m,n}}.
\end{align*}
and suggests to reject the hypothesis $H_0$ in \eqref{Testproblem:IncoVariance} if $\wh{W}_{m,n} > q_{1-\alpha}$, where $q_{1-\alpha}$ is 
the $(1-\alpha)$-quantile of
\begin{align}\label{E:LimitIncoVarStatistic}
	W = \frac{2B(1)}{\left\{ \int_0^1 \int_0^1 \Big[ st\big(tB(s) + sB(t) - 2stB(1)\big) \Big]^2  \, d\nu(s,t) \right\}^{1/2}}.
\end{align}  
As immediate consequence we have that the test is asymptotically consistent for the relevant hypotheses in \eqref{Testproblem:IncoVariance}. The proof works in the same fashion as the proof of Theorem~\ref{T:FrechetVarTest} and we skip it.

\begin{theorem}\label{T:StatW} Let the regularity conditions of Theorem~\ref{Thrm:IncoVarianceJointConvergence} be satisfied. Then 	\begin{align*}
	\lim_{m,n \to \infty} \p( \wh{W}_{m,n} > q_{1-\alpha} ) 
	= \begin{cases}
		1 &, D^2 > \Delta,\\
		\alpha &, D^2 = \Delta,\\
		0 &, D^2 < \Delta.
	\end{cases} 
	\end{align*}
\end{theorem}

\section{Mathematical background: a two-parameter FCLT for U-statistics}\label{Section_Background}

In this section we derive a general functional central limit theorem for (bivariate) U-statistics: 
Let $(M,d)$ denote a metric space and let $h\colon M \times M \to \R$ be a symmetric and measurable kernel.

Let $(X_t)_{t\in\Z}$ be a stationary sequence of $M$-valued random elements defined on some common probability space $(\Omega,\cF,\p)$ with marginal distribution $X_t\sim \p_X$, $t\in\Z$. 
The corresponding U-statistic is defined by
\begin{align}\label{Def:UStatistic}
U_n(h) = \frac{1}{n(n-1)} \sum_{1\le i\le n} \sum_{1\le j \le n \atop j\neq i} h(X_i,X_j) 
= \frac{2}{n(n-1)} \sum_{1\le i< j\le n} h(X_i,X_j). 
\end{align}
Suppose $X'$ is an independent copy of $X_1$, i.e., $X'$ is independent of $X_1$ and has the same law $\p_X$. Then \eqref{Def:UStatistic} is a possible 
estimator for $\theta := \E[h(X_1,X')]$. It is well known that there is a close relation between the U-statistic $U_n(h)$ and UMV-estimators for $\theta$ in the special case of an underlying i.i.d. 
sequence $(X_t)_t$ with absolutely continuous distribution function, we refer to the classical book of Shao \cite{shao2003mathematical}. In extension of the celebrated results of Hoeffding \cite{hoeffding1948StatisticsWithAsympNormalDistribution},
the weak limit of U-statistics has been derived under several weak dependence measures, see the contributions of Eagleson \cite{eagleson1979orthogonal}, Denker \cite{denker1982statistical}, 
Carlstein \cite{carlstein1988degenerate}, Babbel \cite{babbel1989invariance}, Dehling and Wendler 
\cite{dehling2010CLTandBootstrapForUStatistics} and Leucht \cite{leucht2012DegenerateUandVStatistics}.

Our contribution in this direction is a functional CLT for U-statistics for so-called \emph{$L^p$-$m$-approximable processes}, i.e., processes $X=(X_t)_t$ such that
\begin{itemize}
	\item[(i)] each $X_t$ admits a Bernoulli shift representation $X_t=f(\epsilon_t, \epsilon_{t-1},\ldots)$, where the map $f\colon S^{\infty} \to M$ is some measurable 
	function and $(\epsilon_t)_t$ are i.i.d. random elements with values in some measure space $S$;
	\item[(ii)] the $m$-dependent approximation process $(X_t^{(m)})_t$ which is given by 
	\begin{align*}
		X_t^{(m)} = f(\epsilon_t, \ldots, \epsilon_{t-m+1}, \epsilon_{t-m}^{(t)}, \epsilon_{t-m-1}^{(t)},\ldots), \quad t\in\Z,
	\end{align*}
	satisfies 
	\begin{align}\label{E:StandardLpM}
	\sum_{m=1}^{\infty} \E \big[d(X_0,X_0^{(m)})^p\big]^{1/p} < \infty.
	\end{align}
	Here, for each \emph{fixed} $t\in \Z$, the sequence $(\epsilon_j^{(t)})_j$ is an independent copy of the original innovation process $(\epsilon_j)_j$. 
\end{itemize} 

In order to derive the weak (functional) limit of \eqref{Def:UStatistic} we follow Hoeffding's approach and consider the decomposition of the U-statistic in a main and a degenerate part:
\begin{align*}
	U_n(h)= \theta + \frac{1}{n(n-1)} \sum_{1\leq i\not= j\leq n} h_2(X_i,X_j) + \frac{1}{n}\sum_{i=1}^n h_1(X_i) + \frac{1}{n}\sum_{j=1}^n h_1(X_j),
\end{align*}
where 
\begin{align*}
h_1\colon M\to\R, \quad x\mapsto \E[ h(x,X')]-\theta~~\text{with}~~ X'\sim \p_X
\end{align*}
and
\begin{align*}
h_2\colon M\times M \to \R,\quad (x,y)\mapsto h(x,y) - \theta - h_1(x) - h_1(y).
\end{align*}
Define for $0\le s,t\le 1$ and $1\le i\le n$ (for $n\ge 2$)
\begin{align*}
	a_n(s,t) &\coloneqq \frac{1}{n(n-1)} \sum_{i=1}^{\floor{ns}} \sum_{j=1}^{\floor{nt}} \1{i\neq j} \\
	 &=  \frac{1}{n(n-1)} \max\{ (\floor{nt}-1)\floor{ns} , (\floor{ns}-1)\floor{nt} , 0\}
\end{align*}
and
\begin{align*}
	b_{n,i}(t) &\coloneqq \frac{1}{n-1}  \sum_{j=1}^{\floor{nt}} \1{i\neq j} \\
	&= \frac{1}{n-1} \max\Big\{ (\floor{nt}-1) \1{i \le \floor{nt} } + \floor{nt} \1{i > \floor{nt} }, 0 \Big\}.
\end{align*}
The corresponding two-parameter partial sum process $U_n(h)\colon [0,1]^2 \to \R$ is given by 	
\begin{align}
	U_n(h)(s,t)
	&:= \frac{1}{n(n-1)}\sum_{i=1}^{\lf ns\rf}\sum_{j=1\atop i\not= j}^{\lf nt\rf} h(X_i,X_j)\label{Eq:PartialSumProcessOfUn(h)}\\
	\begin{split}\label{Eq:Decomposition1}
	&= a_n(s,t) \theta 
		+ \frac{1}{n(n-1)} \sum_{i=1}^{\lf ns\rf}\sum_{j=1\atop i\not= j}^{\lf nt\rf} h_2(X_i,X_j)\\
		&\quad + \frac{1}{n}\sum_{i =1}^{\lf ns\rf} h_1(X_i) \ b_{n,i}(t)
		+ \frac{1}{n}\sum_{j =1}^{\lf nt\rf} h_1(X_j) \ b_{n,j}(s).
	\end{split}
\end{align}
The symmetric function $h_2$ yields the so-called degenerate part of the $U$-statistic, that is
\begin{align*}
	\E[ h_2(x,X_1) ] 
	&= \E[h(x,X_1)] - \theta - \E[h(x,X')] + \theta = 0 \qquad \forall x\in M,
\end{align*}
because $\E[h_1(X_1)] = \theta -\theta = 0$.
Hence, we expect the functions $h_1$ to contribute most to $U_n(h)$.

In order to derive our FCLT, we need to strengthen the assumptions in two ways (we make this rigorous in Assumption~\ref{AssumptionFCLT} below): On the one hand, 
additionally to the approximation property \eqref{E:StandardLpM}, we assume a specific decay of the sequence $(\E [d(X_0, X_0^{(m)})])_m$.
On the other hand, we require some control on the regularity and growth of the kernel $h$ to insure that the degenerate part in \eqref{Eq:Decomposition1} 
asymptotically vanishes.

\begin{assumption}[\bfseries{p,q,$\varrho$}]\label{AssumptionFCLT} Let $p,q \geq 1$ be H\"older conjugates of each other, i.e., $1/p + 1/q = 1$, and let $\varrho\in (0,1]$.
	\begin{itemize}
	\item[(A1)] The kernel $h$ is $\varrho$-H\"older continuous w.r.t.\ the product metric on $M\times M$, i.e.,
	\begin{align}\label{Cond:L-Continuity_h}
		|h(x,y) - h(x',y')| 
		&\leq L~ d_{M\times M}((x,y), (x',y') )^\varrho  = L~ ( d(x,x') + d(y,y') ) ^\varrho 
	\end{align}
	for some $L> 0$. 
	Moreover, the kernel satisfies the moment condition 
	\begin{align}\label{Cond:Moment_h}
		\sup_{U,V \sim \p_X} \E\big[|h(U,V)|^q\big]^{1/q} =: K(q) < \infty,
	\end{align}
	where the supremum is taken over all couplings $(U,V)$ with each marginal distribution equal to $\p_X$.
	\item[(A2)]	Let $X=(X_t)_t$ be an $M$-valued $L^p$-$m$-approximable process, whose Fr{\'e}chet variance exists, i.e., we have $\E[d(X_t,x)^2] < \infty$ 
	for some $x\in M$. Moreover, $X$ is approximable in the sense that
	\begin{align}\label{Cond:Approximation}
		\sum_{m=1}^{\infty} m~\E \big[d(X_0, X_0^{(m)})^{\varrho p}\big]^{1/p} = \sum_{k=1}^{\infty} \sum_{m\geq k}\E\big[d(X_0,X_0^{(m)})^{\varrho p}\big]^{1/p} < \infty. 
	\end{align}
	\end{itemize}
\end{assumption}

\begin{remark} Assumption~{\bfseries (p,q,$\varrho$)} entails the Hölder continuity of $h_1$ and $h_2$, moreover, the approximation property at a linear rate can be guaranteed under a stronger polynomial rate. In detail we have the following. 
\begin{itemize}
	\item {\it Approximation property.} If $\varrho = 1$, the kernel is even Lipschitz continuous and we can neglect $\varrho$ in \eqref{Cond:Approximation}: the process $X$ is 
	$L^p$-$m$-approximable {\it at a linear rate $m$}.\\
	
	If, however, $0<\varrho < 1$, condition \eqref{Cond:Approximation} holds provided that 
	\begin{align}
		\sum_{m=1}^{\infty} m^{(1+\sigma)/\varrho}~\E \big[d(X_0, X_0^{(m)})^{p}\big]^{1/p} < \infty
	\end{align}
	for some $\sigma>1-\varrho$, i.e., the process $X$ is $L^p$-$m$-approximable {\it at a polynomial rate $m^{(1+\sigma)/\varrho}$}. Indeed, applying Jenssen's and H\"older's 
	inequality in this latter case shows
	\begin{align*}
		&\sum_{m=1}^{\infty} m~\E \big[d(X_0, X_0^{(m)})^{\varrho p}\big]^{1/p} \\
		&\le \sum_{m=1}^{\infty} \left( m^{1+\sigma}~\E \big[d(X_0, X_0^{(m)})^{p}\big]^{\varrho/p} \right)  m^{-\sigma} \\
		&\le \left( \sum_{m=1}^{\infty}  m^{(1+\sigma)/\varrho}~\E \big[d(X_0, X_0^{(m)})^{p}\big]^{1/p} \right)^\varrho \ 
		\left( \sum_{m=1}^{\infty}  m^{-\sigma/(1-\varrho)} \right) ^{1-\varrho}.
	\end{align*}
	\item {\it H\"older continuity.} $h_1$ inherits the $\varrho$-H\"older continuity of the kernel $h$ with the same H\"older constant $L$ because
\begin{align}
	|h_1(x) - h_1(y)| 
	&\le  \E \big[ | h(x,X') - h(y,X') | \big]\nonumber\\
	&\le L~ \E\big[ \| (d(x,y), d(X',X')) \|_r^\varrho \big] \le L~ d(x,y)^\varrho. \label{Eq:LipContinuityH1}
\end{align}
Moreover, the function $h_2$ is $\varrho$-H\"older continuous as well: 
\begin{align}
	&|h_2(x,y) - h_2(x',y')|\nonumber\\
	&\le | h(x,y) - h(x',y') | + |h_1(x) - h_1(x')| + |h_1(y) - h_1(y') |\nonumber\\
	&\le L~\Big( \| (d(x,x'), d(y,y')) \|_r^\varrho + d(x,x')^\varrho + d(y,y')^\varrho \Big)\nonumber\\
	&\le L~\Big( (d(x,x') + d(y,y') )^\varrho + d(x,x')^\varrho + d(y,y')^\varrho \Big)\nonumber\\
	&\le 3L \Big( d(x,x') + d(y,y') \Big)^\varrho, \label{Eq:LipContinuityH2}
\end{align}  
where the last inequality follows from the triangle inequality.
\end{itemize}
\end{remark}

\begin{theorem}[FCLT for $U_n(h)$]\label{Thrm:FCLTforUStatistic}
Granted Assumption \ref{AssumptionFCLT}~$(p,q,\varrho)$ is satisfied for some $p\geq 2$. Then
\begin{align*}
	\Big( \sqrt{n}\big(U_n(h)(s,t)-st\theta\big) ~\big|~ 0\leq s,t\leq 1 \Big) \Rightarrow Y
\end{align*}
in $(D_2,\fD_2)$, where $Y$ is a zero mean Gaussian process given by
\begin{align*}
	Y(s,t) = t \cW_{\Gamma}(s) + s \cW_{\Gamma}(t), \qquad 0\le s,t\le 1,
\end{align*}
with $\cW_{\Gamma} = \sqrt{\Gamma}B$ for a standard Brownian motion $B$ and $\Gamma \coloneqq \sum_k \Cov(h_1(X_0),h_1(X_k))$.
\end{theorem}
\begin{proof}[Proof of Theorem~\ref{Thrm:FCLTforUStatistic}]
We start with the decomposition given in \eqref{Eq:Decomposition1}
	\begin{align*}
	\sqrt{n}U_n(h)(t,s) &= \sqrt{n}a_n(s,t) \theta + \sqrt{n}U_n(h_2)(s,t)\\ 
	& + \frac{1}{\sqrt{n}} \sum_{i =1}^{\lf ns\rf} h_1(X_i) \ b_{n,i}(t)
	+ \frac{1}{\sqrt{n}} \sum_{j =1}^{\lf nt\rf} h_1(X_j) \  b_{n,j}(s). 
	\end{align*}
Clearly, $\sup_{t\ge 2/n} \max_{i\le n} | b_{n,i}(t) - t| \le 3/(n-1)\to 0$ as $n\to \infty$. Also, note that $(h_1(X_t))_t$ is again an $L^2$-$m$-approximable one-dimensional process: We have due to the $\varrho$-H\"older continuity 
of $h_1$ and the condition in \eqref{Cond:Approximation}
\begin{align*}
	\sum_{m=1}^\infty \E[ (h_1(X_0)-h_1(X_0^{(m)}))^2]^{1/2} 
	&\le \sum_{m=1}^\infty L~ \E[ d(X_0,X_0^{(m)})^{2\varrho}]^{1/2}\\
	&\le L \sum_{m=1}^\infty m\ \E[ d(X_0,X_0^{(m)})^{p\varrho}]^{1/p}
	< \infty.
\end{align*}
Moreover, $\E[h_1(X_t)]=0$ and we have  that $\E[ d(X_t,x)^{2} ]<\infty$ for some $x\in M$ by Assumption~\ref{AssumptionFCLT}, thus,
\begin{align*}
	\E[|h_1(X_t)|^2] &\le 2 \E[ |h_1(X_t) - h_1(x)|^2 ] + 2 h_1(x)^2 \\
	&\le 2L^2 \E[ d(X_t,x)^{2\rho} ]  + 2 h_1(x)^2 < \infty.
\end{align*}
Hence, an application of Theorem \ref{T:2ParaDonsker} yields that we have the following functional weak convergence in the Skorokhod space $(D_2,\fD_2)$
\begin{align*}
	\left(\frac{1}{\sqrt{n}} \sum_{i =1}^{\lf ns\rf} h_1(X_i) b_{n,i}(t)
	+ \frac{1}{\sqrt{n}} \sum_{j =1}^{\lf nt\rf} h_1(X_j) \ b_{n,j}(s)\Big|~ 0\leq s,t\leq 1 \right)
	\Rightarrow Y.
\end{align*}
Furthermore observe that $\sup_{s,t\ge 2/n} \sqrt{n}|a_n(s,t) - st| \cdot |\theta| \le 3|\theta|\sqrt{n}/(n-1)\to 0$ as $n\to\infty$. Indeed, if $s\le t$, then $\sqrt{n}(a_n(s,t) - st)$ is bounded from above by
\begin{align*}
	\sqrt{n}\left( \frac{ns(nt-1)}{n(n-1)} -st \right)
	= \sqrt{n} \frac{s(t-1)}{n-1}  \to 0
\end{align*}
also if $s\le t$, then $\sqrt{n}(a_n(s,t) - st)$ bounded from below by 
\begin{align*}
	\sqrt{n}\left( \frac{(ns-1)(nt-2)}{n(n-1)} -st \right)
	= \sqrt{n} \frac{s(t-1)}{n-1} + \sqrt{n} \frac{2/n-s-t}{n-1} \to 0.
\end{align*} 
If $t<s$, the calculations work in a similar fashion.

Thus it remains to show that the degenerate part $Z_n:=\sqrt{n}U_n(h_2)$ converges weakly in $(D_2,\fD_2)$ to the zero process. To that end 
we make use of the convergence criteria laid out in detail in Bickel and Wichura \cite{bickel1971convergence}: We use the corollary given after Theorem~2 combined with the tightness 
condition of Theorem~$3$ in the very paper. More precisely, we verify that
\begin{itemize}
	\item[(1)] the finite-dimensional distributions of $Z_n$ converge to those of the zero process; 
	\item[(2)] $\displaystyle\E[|Z_n(A)| |Z_n(B)| ] \le C |A|^{3/4} |B|^{3/4}$ for all \emph{neighboring} blocks $A,B \subseteq [0,1]^2$ for a certain constant $C\in\R_+$.
\end{itemize}
Here, two blocks $A=( s_1,t_1 ] \times (s_2,t_2]$, $0\leq s_i < t_i \leq 1$, and 
$B=( u_1,v_1 ] \times (u_2,v_2]$, $0\leq u_i < v_i \leq 1$, are said to be neighboring if they have a common edge. 
Adapting the notation in Bickel and Wichura \cite{bickel1971convergence} we write
\begin{align}\label{Eq:2Dincrement}
	X(A) := X(s_1,s_2) - X(t_1,s_2) - X(s_1,t_2) + X(t_1,t_2)
\end{align} 
for the $2$-dimensional increment of a process $X=(X(s,t))_{s,t\in[0,1]} \in D_2$ over a block $A$.\\[5pt]

{\it Statement (1)}: Observe that for $m\in\N$, $\alpha_1,\ldots,\alpha_m\in\R$ and $\epsilon>0$ we have as a consequence of
Corollary~\ref{T:VarDegenerateU} and Markov's inequality
\begin{align*}
	&\p\Big( \Big| \sum_{k=1}^m \alpha_k Z_n(s_k,t_k) \Big| \ge \epsilon \Big)\le \epsilon^{-1} \ \E\Big[ \Big| \sqrt{n} \sum_{k=1}^m \alpha_k U_n(h_2)(s_k,t_k) \Big|  \Big] \\
	&\le \epsilon^{-1} \ \E\Big[ \Big| \sqrt{n} \sum_{k=1}^m \alpha_k U_n(h_2)(s_k,t_k) \Big|^2 \Big]^{1/2}  
	\le \frac{\sqrt{n}}{\epsilon} \  \ \sum_{k=1}^m |\alpha_k| \ \E[ |  U_n(h_2)(s_k,t_k) |^2 ]^{1/2} \\
	&= \frac{\sqrt{n}}{\epsilon} \ \sum_{k=1}^m |\alpha_k| \ 
		 \E \Big[ \Big( \frac{1}{n(n-1)} \sum_{i=1}^{\lf ns_k\rf}\sum_{j=1\atop i\not= j}^{\lf nt_k\rf} h_2(X_i,X_j)  \Big)^2\Big]^{1/2} \\
	&\leq  \frac{\sqrt{n}}{\epsilon} \ \sum_{k=1}^m 
	\frac{ |\alpha_k|}{n(n-1)} \left(\sum_{i,j=1\atop i\not= j}^n \sum_{u,v=1\atop u\not= v}^n \big|\E[ h_2(X_i,X_j)h_2(X_u,X_v) ]\big| \right)^{1/2}	 \lesssim \frac{\|\alpha\|_1 }{\epsilon \sqrt{n}}.
\end{align*}
Hence it follows from the Cram\'er-Wold device that the finite dimensional 
distributions of $Z_n$ converge weakly to those of the zero process.\\

{\it Statement (2)}: Relying on the Cauchy-Schwarz inequality, it is sufficient to verify the moment condition 
\begin{align}\label{Eq:MomentCondition1}
	\E[ | Z_n(A) |^2 ] \le C |A|^{3/2}
\end{align} 
for all blocks $A$ in the unit square. This condition is verified in Theorem~\ref{T:MomentCondition}.
\end{proof}

The proofs of the next two theorems are postponed to Section~\ref{Subsection_DetailsUStats}.
\begin{theorem}[Two-parameter Donsker type FCLT]\label{T:2ParaDonsker}
Let $Z = (Z_t)_{t\in\Z}$ be a real-valued $L^2$-$m$-approximable process such that $\E[Z_t]=0$ and $\E[Z_t^2] < \infty$. 
Define the partial-sum process
\begin{align}\label{E:2ParaDonsker0}
	Y_n\colon [0,1]^2\to\R,\quad (s,t) \mapsto \frac{t}{\sqrt{n}} \sum_{i=1}^{\floor{ns}} Z_i + \frac{s}{\sqrt{n}} \sum_{i=1}^{\floor{nt}} Z_i.
\end{align}
Then $Y_n\Rightarrow Y$ in the two-dimensional Skorokhod space $(D_2,\fD_2)$, where the Gaussian process $Y$ is given by 
\begin{align*}
	Y(s,t) = t \cW_{\Gamma}(s) + s \cW_{\Gamma}(t) 
\end{align*}
with $\cW_{\Gamma} = \sqrt{\Gamma}B$, for a standard Brownian motion $B$ and $\Gamma = \sum_{h\in\Z} \Cov(Z_0,Z_h)$.
\end{theorem}

\begin{theorem}[Moment condition for $U_n(h_2)$]\label{T:MomentCondition}
Granted Assumption~\ref{AssumptionFCLT} (p,q,$\rho$) is satisfied there is a universal constant $C\in\R_+$ such that for all blocks $A = (s,u]\times (t,v] \subset [0,1]^2$ 
\begin{align}\label{E:MomentCondition0}
	\E\left[ \left|\sqrt{n} \ U_n(h_2)(A)  \right|^2 \right] \le C |A|^{3/2},
\end{align}
where $U_n(h_2)(A)$ is the $2$-dimensional increment of $U_n(h_2)$ over $A$ in the sense of \eqref{Eq:2Dincrement}, i.e., 
\begin{align*}
	U_n(h_2)(A) = U_n(h_2)(u,v)+U_n(h_2)(s,t) - U_n(h_2)(u,t) - U_n(h_2)(s,v).
\end{align*}
\end{theorem}

The proof of Theorem~\ref{T:MomentCondition} yields a straightforward corollary:

\begin{cor}[Second moment of $U_n(h_2)$]\label{T:VarDegenerateU}
Granted Assumption~\ref{AssumptionFCLT} (p,q,$\rho$) is satisfied there is a universal constant $C\in\R_+$, such that
\begin{align*}
	\sum_{1\leq i<j \leq n,\atop 1\leq k<\ell \leq n} | \E[ h_2(X_i,X_j) h_2(X_k,X_\ell) ] | \le C n^2.
\end{align*}
\end{cor}

\section{Mathematical details}\label{Section_MathematicalDetails}

\subsection{Two sample test for Fr{\'e}chet variances}

\begin{proof}[Proof of Theorem~\ref{T:FrechetVarianceJointConvergence}]	
The proof is divided into two steps.

{\it Step 1.} We verify the assumptions of Theorem~\ref{Thrm:WeakConvergenceTwoSampleRelevantDifference_General} for 
$X_i := W_r^2(\pd(\cX_i),\mu_\cX)$, resp., $Y_i := W_r^2(\pd(\cY_i),\mu_\cY)$. We show that $(X_i)_i$ and $(Y_i)_i$ inherit the $L^p$-$m$-approximation property from the underlying point cloud processes 
$(\cX_t)_t$ and $(\cY_t)_t$, which then implies a strong law of large numbers as well as an invariance principle for the corresponding partial sum processes.
Again, we only present the calculations for $(X_i)_i$: To that end, let 
	\begin{align*}
	X_i^{(m)} := W_r^2(\pd(\cX_i^{(m)}),~\mu_\cX), \qquad m\in\N
	\end{align*}
denote the corresponding approximation process, cf. Assumption \ref{A:FrechetVariances}(r) part (2).
Then, by the reverse triangle inequality,  
\begin{align*}
	&\E[| X_0 - X_0^{(m)} |^2]
	= \E[ |W_r^2(\pd(\cX_0), \mu_{\cX}) - W_r^2(\pd(\cX_0^{(m)}), \mu_{\cX})|^2 ]\\
	& \le \E\Big[ \{ 2 W_r(\pd(\cX_0), \mu_\cX) + W_r(\pd(\cX_0^{(m)}), \pd(\cX_0) ) \}^2 \\
	&\quad\qquad\qquad \cdot | W_r(\pd(\cX_0), \mu_\cX) - W_r(\pd(\cX_0^{(m)}), \mu_\cX) |^2 \Big]\\
	&\le \E\Big[ \{ 4 W_r^2(\pd(\cX_0), \mu_\cX) + 2 W_r^2(\pd(\cX_0^{(m)}), \pd(\cX_0) ) \} \cdot  W_r^2(\pd(\cX_0),\pd(\cX_0^{(m)})) \Big] \\
	&= 4 \E[ W_r^4(\pd(\cX_0), \mu_\cX)]^{1/2} \cdot \E[ W_r^4(\pd(\cX_0),\pd(\cX_0^{(m)})) ]^{1/2} \\
	&\quad + 2 \E[ W_r^4(\pd(\cX_0),\pd(\cX_0^{(m)})) ].
\end{align*}
Consequently, (since $\sqrt{a+b}\le \sqrt{a}+\sqrt{b}$ for $a,b\ge 0$),
\begin{align*}
	\sum_{m \ge 1} \E[| X_0 - X_0^{(m)} |^2]^{1/2}
	&\leq 2\E[ W_r^4(\pd(\cX_0), \mu_\cX)]^{1/4} \sum_{m \ge 1}\E[ W_r^4(\pd(\cX_0),\pd(\cX_0^{(m)})) ]^{1/4}\\
	&\quad+ \sqrt{2}\sum_{m \ge 1} \E[ W_r^4(\pd(\cX_0),\pd(\cX_0^{(m)})) ]^{1/2}.
\end{align*}
By assumption $\E[ W_r^4(\pd(\cX_0), \mu_\cX)]$ is finite and thus it remains to establish 
\begin{align*}
	\sum_{m \ge 1}\E[ W_r^4(\pd(\cX_0),\pd(\cX_0^{(m)})) ]^{1/4} < \infty.
\end{align*}
To that end we use Lemma~\ref{L:PDandHD}, which implies the following fundamental inequality between the Wasserstein distance $W_r$, resp., $W_B$ of two diagrams 
$\pd(U),\pd(V)$ and the Hausdorff distance $d_H$ of the underlying point clouds $U,V$: If $r<\infty$, then
\begin{align*}
	W_r(\pd(U), \pd(V) ) &\le W_B( \pd(U), \pd(V) ) \cdot \max\{ \# \pd(U),\# \pd(V) \}^{1/r} \\
	&\le 2 d_H(U,V) \cdot \max\{ \# \pd(U), \# \pd(V) \}^{1/r}.
\end{align*}
Thus, we have for any H\"older conjugates $p,q\geq 1$
\begin{align*}
	&\E[ W_r^4(\pd(\cX_0),\pd(\cX_0^{(m)})) ]^{1/4}\\
	&\leq 2\ \E[ d_H^{4}(\cX_0,\cX_0^{(m)}) \cdot \max\{ \# \pd(\cX_0), \# \pd(\cX_0^{(m)}) \}^{4/r} ]^{1/4}\\
	&\leq 2\ \E[ d_H^{4p}(\cX_0,\cX_0^{(m)}) ]^{1/(4p)}~\E[\max\{ \# \pd(\cX_0), \# \pd(\cX_0^{(m)}) \}^{4q/r}]^{1/(4q)}\\
	&\leq 2^{1+1/(4q)}\ \E[ d_H^{4p}(\cX_0,\cX_0^{(m)}) ]^{1/(4p)}~\E[ \# \pd(\cX_0)^{4q/r} ])^{1/(4q)}.
\end{align*}
Now, if we take $p=r/4$ (recall that $r> 4$ is characterized in Assumption~\ref{A:FrechetVariances}) and accordingly $q=r/(r-4)$, then 
\begin{align*}
	&\sum_{m \ge 1}\E[ W_r^4(\pd(\cX_0),\pd(\cX_0^{(m)})) ]^{1/4}\\
	&\leq 2^{5/4-1/r}\ \E[ \# \pd(\cX_0)^{4/(r-4)} ]^{1/4-1/r} \sum_{m \ge 1} \E[ d_H^{r}(\cX,\cX_0^{(m)}) ]^{1/r}
\end{align*}
and the RHS is finite by assumption. This shows that the one-dimensional processes $X_i = W_r^2(\pd(\cX_i),\mu_\cX)$ and $Y_i = W_r^2(\pd(\cY_i),\mu_\cY)$ enjoy
the $L^2$-$m$-approximation property.

Furthermore by assumption $\E[X_0] = \E[ W_r^2(\pd(\cX_0),\mu_{\cX}) ] < \infty$ as well as $\E[X_0^2] = \E[ W_r^4(\pd(\cX_0),\mu_{\cX}) ] < \infty$ and 
the same is true for $(Y_i)_i$. Hence the conditions of Theorem A.1 in Aue et al. \cite{aue2009break} are satisfied for the centered processes
$(X_i-\E[X_0])_i$, resp., $(Y_i-\E[Y_0])_i$ and therefore (w.r.t.\ the Skorokhod topology on $D_1$)
\begin{align*}
	\frac{1}{\sqrt{m}} \sum_{i=1}^{\floor{s m}} (X_i - \E[X_0]) &\Rightarrow \sqrt{\Gamma_\cX} B_1(s),\\
	\frac{1}{\sqrt{n}} \sum_{i=1}^{\floor{s n}} (X_i - \E[X_0]) &\Rightarrow \sqrt{\Gamma_\cY} B_2(s), \quad s \in [0,1],
\end{align*}
where $B_1$, $B_2$ are independent standard Brownian motions and
\begin{align*}
	\Gamma_\cX = \sum_{i\in\Z} \Cov( X_0,X_i ),\quad \Gamma_{\cY} = \sum_{i\in\Z} \Cov( Y_0,Y_i ).
\end{align*}
Moreover,
\begin{align*}
	\frac{1}{m} \sum_{i=1}^{ m} X_i \to \E[X_0] \quad a.s. \text{ and } \frac{1}{n} \sum_{i=1}^{ n} Y_i \to \E[Y_0] \quad a.s.
\end{align*}
by Corollary \ref{Cor:SLLNforL^2-m-processes}.
Altogether, we see that $(X_i)_i$ and $(Y_i)_i$ fulfill all requirements in Theorem \ref{Thrm:WeakConvergenceTwoSampleRelevantDifference_General}.

{\it Step 3.} Finally, we verify that the weak convergence carries over to the original processes. To that end, define
\begin{align*}
	\wt{\operatorname{Var}}(\cX)(s) = \frac{1}{\lf ms \rf} \sum_{i=1}^{\lf ms \rf} X_i = \frac{1}{\lf ms \rf} \sum_{i=1}^{\lf ms \rf} W_r^2(\pd(\cX_i),\mu_\cX), \qquad s \in[0,1],\\
	\wt{\V}(\cY)(s) = \frac{1}{\lf ns \rf} \sum_{i=1}^{\lf ns \rf} Y_i = \frac{1}{\lf ns \rf} \sum_{i=1}^{\lf ns \rf} W_r^2(\pd(\cY_i),\mu_\cY),\qquad s \in[0,1],
\end{align*}
Next, set
$$
	\wt D_{m,n}(s) = s(\wt{\operatorname{Var}}(\cX)(s) - \wt{\operatorname{Var}}(\cY)(s)), \qquad s\in [0,1],
$$
and $\wt D_{m,n} = \wt D_{m,n}(1)$ as well as
$$
	\wt V_{m,n} = \Big\{ \int_0^1 ( \wt D_{m,n}^2 (s) - s^2 \wt D_{m,n}^2)^2 \nu(\diff s) \Big\}^{1/2}.
$$
Then
$$
	\sqrt{m+n} (\wh V_{m,n} - \wt V_{m,n}) \le  \sqrt{2} \ \sqrt{m+n} \max_{s\in [0,1]} |\wh D_{m,n}^2 (s) - \wt D_{m,n}^2 (s)|
$$
by the reverse triangle-inequality and
$$
	\sqrt{m+n} \Big( (\wh D^2_{m,n} - D^2) - (\wt D^2_{m,n} - D^2) \Big) \le \sqrt{m+n}  \max_{s\in [0,1]} |\wh D_{m,n}^2 (s) - \wt D_{m,n}^2 (s)| .
$$
Now, we have
\begin{align*}
	&\sqrt{m} \ \max_{s\in [0,1]} |\wh D_{m,n}^2(s) - \wt D_{m,n}^2(s)| \\
	 &\le   \Big\{ \max_{s\in [0,1]} \sqrt{m} s\cdot |\wt{\operatorname{Var}}(\cX)(s) - \wh{\operatorname{Var}}(\cX)(s) | + \max_{s\in [0,1]} \sqrt{m} s\cdot |\wt{\operatorname{Var}}(\cY)(s) - \wh{\operatorname{Var}}(\cY)(s) | \Big\} \\
	&\quad \cdot   \Big\{ \max_{s\in [0,1]} s\cdot |\wt{\operatorname{Var}}(\cX)(s) + \wh{\operatorname{Var}}(\cX)(s) | + \max_{s\in [0,1]} s\cdot |\wt{\operatorname{Var}}(\cY)(s) + \wh{\operatorname{Var}}(\cY)(s) | \Big\}. 
\end{align*}
Consequently, the desired weak convergence for the random vector $\sqrt{m+n}( \wt D_{m,n}^2 - D^2, \wt V_{m,n} )$ follows if we show: For each $\delta,\epsilon>0$, there is a $C\in\R_+$ and $m_0\in\N$ such that
\begin{align}
	&\p\Big( \max_{s\in [0,1]} s \cdot |\wt{\operatorname{Var}}(\cX)(s) + \wh{\operatorname{Var}}(\cX)(s) | > C \Big) \le \epsilon, \label{E:FrechetTest-1} \\
	&\p\Big( \max_{s\in [0,1]} s \cdot |\wt{\operatorname{Var}}(\cY)(s) + \wh{\operatorname{Var}}(\cY)(s) | > C \Big) \le \epsilon \qquad \forall m\ge m_0 \nonumber
\end{align}
and
\begin{align}
	&\p\Big( \max_{s\in [0,1]} \sqrt{m} s \cdot |\wt{\operatorname{Var}}(\cX)(s) - \wh{\operatorname{Var}}(\cX)(s) | > \delta \Big) \le \epsilon, \label{E:FrechetTest-2}\\
	&\p\Big( \max_{s\in [0,1]}  \sqrt{m} s \cdot |\wt{\operatorname{Var}}(\cY)(s) - \wh{\operatorname{Var}}(\cY)(s) | > \delta \Big) \le \epsilon \qquad \forall m\ge m_0. \nonumber
\end{align}
Clearly, it is sufficient to focus on the case for $\cX$. We treat \eqref{E:FrechetTest-1} first and \eqref{E:FrechetTest-2} afterwards. To this end, we bound above the summands of $\wt{\operatorname{Var}}(\cX)(s) + \wh{\operatorname{Var}}(\cX)(s)$ by
\begin{align*}
	&| W_r^2(\pd(\cX_i), \wh{\mu}_\cX) + W_r^2(\pd(\cX_i), \mu_\cX) | \le  ( 3W^2_r(\pd(\cX_i), \mu_\cX) + 2W^2_r(\wh\mu_{\cX}, \mu_\cX) ).
\end{align*}
Thus, the left-hand side of the event in \eqref{E:FrechetTest-1} is at most
\begin{align*}
	&\max_{s\in [0,1]} s \cdot  \frac{1}{\floor{ms}} \sum_{i=1}^{\floor{ms}} \Big( 3W_r^2( \pd(\cX_i),\mu_\cX ) + 2W_r^2(\wh\mu_\cX,\mu_\cX) \Big)   \\
	&\lesssim \frac{1}{m}  \sum_{i=1}^{m} \Big( 3W_r^2( \pd(\cX_i),\mu_\cX ) + 2W_r^2(\wh\mu_\cX,\mu_\cX) \Big).   
\end{align*}
Applying Corollary~\ref{Cor:SLLNforL^2-m-processes} to the $L^2$-$m$-approximable sequence $ (W_r^2( \pd(\cX_i),\mu_\cX ))_i $ and relying on the fact that $W_r^2(\wh\mu_\cX,\mu_\cX) \to 0$ in probability ($m\to\infty$), verifies \eqref{E:FrechetTest-1}.

The left-hand side of the event in the probability in \eqref{E:FrechetTest-2} is at most
\begin{align}
	&\max_{0\le s\le m^{-1/2} } \frac{1}{\sqrt{m}}  \sum_{i=1}^{\floor{ms} }  \Big( W_r(\pd(\cX_i),\mu_\cX) + W_r(\mu_\cX,\wh\mu_\cX)  \Big) W_r(\wh\mu_\cX,\mu_\cX) \nonumber \\
	&\quad + \max_{ m^{-1/2} \le s \le 1 }  \frac{1}{\sqrt{m}} \Big| \sum_{i=1}^{\floor{ms} }  W^2_r(\pd(\cX_i),\mu_\cX) - W^2_r(\pd(\cX_i),\wh\mu_\cX)  \Big| \nonumber \\
	&\le  \frac{1}{\sqrt{m}} \sum_{i=1}^{\floor{\sqrt{m} } }  \Big( W_r(\pd(\cX_i),\mu_\cX) + W_r(\mu_\cX,\wh\mu_\cX)  \Big) W_r(\wh\mu_\cX,\mu_\cX) \label{E:FrechetTest-3} \\
	&\quad+ \max_{\floor{\sqrt{m}}\le k\le m} \frac{1}{\sqrt{k}} \Big| \sum_{i=1}^{k }  W^2_r(\pd(\cX_i),\mu_\cX) - W^2_r(\pd(\cX_i),\wh\mu_\cX)  \Big| \label{E:FrechetTest-4}
\end{align}
The term in \eqref{E:FrechetTest-3} vanishes by the strong law of large numbers from  Corollary~\ref{Cor:SLLNforL^2-m-processes} and the assumption that $ W_r(\wh\mu_\cX,\mu_\cX) \to 0$ in probability for $m\to\infty$.

The term in \eqref{E:FrechetTest-4} vanishes by Assumption~\ref{A:FrechetVariances}.This verifies \eqref{E:FrechetTest-2}.
\end{proof}

\begin{theorem}\label{Thrm:WeakConvergenceTwoSampleRelevantDifference_General}
Let $X_1,\ldots,X_m$ and $Y_1,\ldots,Y_n$ be two samples stemming from independent and real-valued stationary processes 
$(X_t)_t$, resp., $(Y_t)_t$ with mean $\E[X_0]$, resp., $\E[Y_0]$. Furthermore assume that the partial sum processes $\ol{X}_m(s) = \sum_{i=1}^{\floor{ms}} X_i$
and $\ol{Y}_n(s) = \sum_{i=1}^{\floor{ns}} Y_i$, $s \in [0,1]$, satisfy a SLLN, i.e., almost surely
\begin{align}
\label{E:MaxConvProp}
	& \frac{1}{m}~ \ol{X}_m(1) \to  \E[X_0] \quad \text{ and } \quad \frac{1}{n}~\ol{Y}_n(1) \to  \E[Y_0]
\end{align}
as well as an invariance principle in the Skorokhod topology on $D_1 = D([0,1])$, namely, 
\begin{align*}
	H_{m} := \frac{1}{\sqrt{m}} (\ol{X}_m(\cdot) - \floor{ m \cdot} \E[X_0]) &\Rightarrow \sqrt{\Gamma_X} B_1,
	\quad  \Gamma_X = \sum_{h\in\Z} \Cov(X_0,X_h);\\
	J_{n} := \frac{1}{\sqrt{n}} (\ol{Y}_n(\cdot) - \floor{ n \cdot} \E[Y_0]) &\Rightarrow \sqrt{\Gamma_Y} B_2,
	\quad  \Gamma_Y = \sum_{h\in\Z} \Cov(Y_0,Y_h),
\end{align*}
where $B_1, B_2$ are independent standard Brownian motions. Let $D = \E[X_0] - \E[Y_0]$ and define
\begin{align*}
	\wt{D}_{m,n}(s) &= s \left( \frac{1}{\floor{ms} } \ol{X}_m(s) - \frac{1}{\floor{ns} } \ol{Y}_n(s)  \right), \quad s\in [0,1],\\
	\wt{V}_{m,n} &= \left\{ \int_0^1 \big(\wt{D}_{m,n}^2(s) - s^2\wt{D}_{m,n}^2(1) \big)^2 \, \nu(\diff s) \right\}^{1/2},
\end{align*}
where we use the convention $\ol X_m(s)/\floor{ms} = 0$ for $s<1/m$.
Suppose there is a constant $\tau \in (0,1)$ such that $\lim_{m,n \to \infty} m/(m+n) = \tau$. Then (w.r.t.\ the Euclidean metric on $\R^2$)
\begin{align*}
	\sqrt{m+n}\Big( \wt{D}_{m,n}^2(1) - D^2,~\wt{V}_{m,n}\Big) \Rightarrow 
	\left( \xi B(1),~ \xi \left\{ \int_0^1 s^2 \big(B(s) - s B(1)\big)^2 \, \nu(\diff s) \right\}^{1/2} \right),
\end{align*}
where $B$ is a standard Brownian motion and $\xi = 2\sqrt{\frac{\Gamma_X}{\tau} + \frac{\Gamma_Y}{1-\tau} }(\E[X_0]-\E[Y_0])$.
\end{theorem}

\begin{proof}[Proof of Theorem~\ref{Thrm:WeakConvergenceTwoSampleRelevantDifference_General}]
As in the formulation of the theorem let $B$, $B_1$, $B_2$ denote independent standard Brownian motions. The proof is divided into several main steps.

{\it Step 1.} We start with rewriting $\sqrt{m+n}( \wt{D}_{m,n}^2(1) - D^2,\wt{V}_{m,n})$ in a more suitable way. Writing out the first term yields
\begin{align}
	\sqrt{m+n}(\wt{D}_{m,n}^2(1) - D^2) 
	&= \sqrt{m+n}(\wt{D}_{m,n}(1) - D)(\wt{D}_{m,n}(1) + D) \nonumber \\
	&= \left( \sqrt{\frac{m+n}{m}} H_m(1) - \sqrt{\frac{m+n}{n}} J_n(1) \right) \nonumber \\
	&\quad \times \left( \frac{1}{m} \ol{X}_m(1) - \frac{1}{n} \ol{Y}_n(1) + \E[X_0] - \E[Y_0]  \right) =: Z_{m,n}. \label{E:WeakConvergenceTwoSample-1}
\end{align}
The second term satisfies
\begin{align}
	\sqrt{m+n}~\wt{V}_{m,n}
	= \left\{ \int_0^1 (m+n)\big(\wt{D}_{m,n}^2(s) - s^2\wt{D}_{m,n}^2(1) \big)^2 \, \nu(\diff s) \right\}^{1/2} \label{E:WeakConvergenceTwoSample-1b}
\end{align}
and its integrand equals
\begin{align}
	&(m+n)\big(\wt{D}_{m,n}^2(s) - s^2\wt{D}_{m,n}^2(1) \big)^2 \nonumber\\ 
	&= (m+n)\big( \wt{D}_{m,n}^2(s) - s^2D^2 - s^2( \wt{D}_{m,n}^2(1) - D^2 ) \big)^2 \nonumber\\
	&= \Big( \sqrt{m+n} \ (\wt{D}_{m,n}(s) - s D)(\wt{D}_{m,n}(s) + s D ) 
	- s^2\sqrt{m+n} \ ( \wt{D}_{m,n}^2(1) - D^2 )\Big)^2 \nonumber\\
	&= \big(\wt Z_{m,n}(s) - s^2 Z_{m,n}\big)^2.\label{E:WeakConvergenceTwoSample-5}
\end{align}
Here we defined the first term in the last curly parentheses by
\begin{align}
	&\wt Z_{m,n}(s) 
	:= \sqrt{m+n} \ (\wt{D}_{m,n}(s) - s D)(\wt{D}_{m,n}(s) + s D) \nonumber \\
	&=  s^{1-\gamma} \ \left\{ \frac{ ms/\lf ms \rf }{\sqrt{m/(m+n)}} \frac{\ol{X}_m(s) - \lf ms \rf \E[X_0]}{\sqrt{m}}
	- \frac{ ns / \lf ns \rf }{\sqrt{n/(m+n)}} \frac{\ol{Y}_n(s) - \lf ns \rf \E[Y_0]}{\sqrt{n}} \right\} \nonumber \\
	&\quad \times s^{\gamma} \ \Biggl\{ \frac{1}{ \lf ms \rf} \sum_{i=1}^{\lf ms \rf} X_i + \E[X_0] 
	- \frac{1}{ \lf ns \rf} \sum_{i=1}^{\lf ns \rf} Y_i -  \E[Y_0] \Biggl\}  \nonumber \\
	&=: \wt Z_{1,m,n}(s) \times  \wt Z_{2,m,n}(s). \label{E:WeakConvergenceTwoSample-6}
\end{align}

{\it Step 2.} Next, in view of \eqref{E:WeakConvergenceTwoSample-1}, let us introduce the functional
\begin{align*}
	F_1\colon D_1 \times D_1 \to \R,\quad (u,v) \mapsto 2(\E[X_0] - \E[Y_0])\Big(\frac{ u(1)}{\sqrt{\tau}} - \frac{v(1)}{\sqrt{1-\tau}} \Big). 
\end{align*}
Then, as an immediate consequence of \eqref{E:MaxConvProp}, we see that 
\begin{align}
	&Z_{m,n} - F_1(H_m,J_n) \nonumber \\ 
	&= \left\{\sqrt{\frac{m+n}{m}} \left( \frac{1}{m} \ol X_m(1) - \frac{1}{n} \ol Y_n(1) + \E[X_0] - \E[Y_0] \right) 
		- \frac{2(\E[X_0] - \E[Y_0])}{\sqrt{\tau}} \right\} H_m(1) \nonumber\\
	&\quad+  \left\{\sqrt{\frac{m+n}{n}} \left( \frac{1}{m} \ol X_m(1) - \frac{1}{n} \ol Y_n(1) + \E[X_0] - \E[Y_0] \right) 
	- \frac{2(\E[X_0] - \E[Y_0])}{\sqrt{1-\tau}} \right\} J_n(1) \nonumber \\
	&\to	0 \label{E:WeakConvergenceTwoSample15}
\end{align}
in probability. Since $F_1$ is continuous with respect to (any) 
product metric on $D_1\times D_1$, we also have by the invariance principle 
\begin{align*}
	F_1(H_m,J_m)~\Rightarrow~ &F_1(\sqrt{\Gamma_X} B_1, \sqrt{\Gamma_Y} B_2)\\
	&= 2(\E[X_0] - \E[Y_0])\Big(\frac{ \sqrt{\Gamma_X}B_1(1)}{\sqrt{\tau}} - \frac{\sqrt{\Gamma_Y}B_2(1)}{\sqrt{1-\tau}} \Big)\\
	&\stackrel{\cL}{=} 2(\E[X_0] - \E[Y_0]) \sqrt{ \frac{\Gamma_X}{\tau} + \frac{\Gamma_Y}{1-\tau} } B(1).
\end{align*}
Consequently, an application of Slutsky's theorem yields $Z_{m,n} \Rightarrow \xi B(1)$.\\

{\it Step 3.} In order to derive the weak limit of \eqref{E:WeakConvergenceTwoSample-1b} we start by considering the integrand first. In view of 
\eqref{E:WeakConvergenceTwoSample-5} and \eqref{E:WeakConvergenceTwoSample-6}, let us introduce the operator
\begin{align*}
	&G_{1,1}\colon D_1\times D_1 \to D_1, \quad (u,v) \mapsto \Big( s^{1-\gamma}\Big( \frac{u(s)}{\sqrt{\tau}} - \frac{v(s)}{\sqrt{1-\tau}} \Big) ~\Big|\, s\in[0,1] \Big),
\end{align*}
and the function
\begin{align*}
	&g_{1,2}\colon [0,1] \to\R, \quad s \mapsto   2s^{\gamma}(\E[X_0] -\E[Y_0])
\end{align*}
for some $\gamma \in (0,1)$ arbitrary but fixed. Then we define the operator $G_1 := G_{1,1} \cdot g_{1,2}$ and finally set
\begin{align*}
	F_2 \colon D_1\times D_1\to D_1, \quad (u,v)\mapsto  (G_1(u,v)-(\cdot)^2 F_1(u,v))^2.
\end{align*}
Observe that $F_2$ is discontinuous only on $\mathrm{Dis}(F_2) = D_1\backslash\cC \times D_1\backslash \cC$, where $\cC$ denotes the set of continuous functions on the interval $[0,1]$. Since by assumption
\begin{align*}
	(H_m, J_n) ~\Rightarrow~(\sqrt{\Gamma_X} B_1, \sqrt{\Gamma_Y} B_2)
\end{align*}
and $\p((\sqrt{\Gamma_X} B_1, \sqrt{\Gamma_Y} B_2) \in \mathrm{Dis}(F_2)) = 0$, we can therefore apply the continuous mapping theorem to obtain 
\begin{align}\label{E:WeakConvergenceTwoSample-8}
	F_2(H_m,J_n) ~\Rightarrow~
	&F_2(\sqrt{\Gamma_X} B_1, \sqrt{\Gamma_Y} B_2).
\end{align}
Next, we show that $F_2(H_m,J_n)$ approximates the asymptotic law of the integrand \eqref{E:WeakConvergenceTwoSample-5} in the sense that
\begin{align}\label{E:WeakConvergenceTwoSample-3}
	d_1\big((\wt Z_{m,n} - (\cdot)^2 Z_{m,n})^2,~ F_2(H_m,J_n)  \big) \to 0
\end{align}
in probability, where $d_1$ denotes the Skorokhod metric on $D_1$, cf. \eqref{SkorokhodMetric} for an exact definition. The Skorokhod metric $d_1$ satisfies 
\begin{align*}
	d_1(u,v) < \delta \text{ implies } d_1(u^2,v^2) < (1+\|u\|_{\infty} + \|v\|_{\infty}) \delta
\end{align*}
for any $u,v\in D_1$ and $\delta > 0$. So, in order to verify \eqref{E:WeakConvergenceTwoSample-3}, it is sufficient that
\begin{align}\label{E:WeakConvergenceTwoSample-4}
	d_1\big(\wt Z_{m,n} - (\cdot)^2 Z_{m,n},~ G_1(H_m,J_n)-(\cdot)^2 F_1(H_m,J_n)  \big) \to 0 \text{ in probability.}
\end{align}
To that end, we first use the triangle inequality to bound the left-hand side in \eqref{E:WeakConvergenceTwoSample-4} by 
\begin{align}\label{E:WeakConvergenceTwoSample-7}
	\begin{split}
	&d_1(\wt Z_{m,n} - (\cdot)^2 Z_{m,n},~ G_1(H_m,J_n) - (\cdot)^2 Z_{m,n} )\\
	&+ d_1(G_1(H_m,J_n) - (\cdot)^2 Z_{m,n},~ G_1(H_m,J_n)-(\cdot)^2 F_1(H_m,J_n)  ).
	\end{split}
\end{align}
We begin with the first term in \eqref{E:WeakConvergenceTwoSample-7}. We have by the definition of the Skorokhod metric
\begin{align*}
	&d_1(\wt Z_{m,n} - (\cdot)^2 Z_{m,n},~ G_1(H_m,J_n) - (\cdot)^2 Z_{m,n} )\\
	&= \inf_{\lambda \in \Lambda} \Big\{ \max\Big\{ \sup_s |\wt Z_{m,n}(s) - s^2 Z_{m,n} - G_1(H_m,J_n)(\lambda(s))+\lambda(s)^2 Z_{m,n} |,\\ 
	&\qquad	\sup_s |\lambda(s) -s|\Big\} \Big\}\\
	&\leq \inf_{\lambda \in \Lambda} \Big\{ 2(2|Z_{m,n}| + 1) \max\Big\{ \sup_s | \wt Z_{m,n}(s) - G_1(H_m,J_n)(\lambda(s)) |,~  \sup_s |\lambda(s) -s|\Big\} \Big\}\\
	&=2(2|Z_{m,n}| + 1)\, d_1( \wt Z_{m,n},~ G_1(H_m,J_n) ).
\end{align*}
In the second step we have already established $Z_{m,n} \Rightarrow \xi B(1)$. 	Furthermore, it follows that $d_1( \wt Z_{m,n}, G_1(H_m,J_n) ) \rightarrow 0$ in probability. Indeed, we have
	\begin{align*}
		d_1\big( \wt Z_{m,n}, \ G_1(H_m,J_n) \big) &= d_1( \wt Z_{1,m,n} \ \wt Z_{2,m,n}, \ G_{1,1}(H_m,J_n) \ g_{1,2} ) \\
		&\le \sup_{s\in[0,1]} \Big| \wt Z_{1,m,n}(s) -  G_{1,1}(H_m,J_n)(s) \Big|  \  \sup_{s\in [0,1]} |\wt Z_{2,m,n} (s)| \\
		&\quad +  \sup_{s\in[0,1]}  |  G_{1,1}(H_m,J_n)(s) |  \  \sup_{s\in[0,1]}  \Big| \wt Z_{2,m,n}(s) - g_{1,2}(s) \Big|.
	\end{align*}
On the one hand, 
	\begin{align*}
	\left( \sup_{s\in [0,1]} |\wt Z_{2,m,n} (s)| \right)_{m,n}  \text{ and } \left( \sup_{s\in[0,1]}  |  G_{1,1}(H_m,J_n)(s) |  \right)_{m,n}
	\end{align*}
	are tight real-valued sequences. On the other hand, it is true that in probability
	\begin{align}
		&\sup_{s\in [0,1]} |\wt Z_{1,m,n}(s) -  G_{1,1}(H_m,J_n)(s)| \to 0, \label{E:WeakConvergenceTwoSample12} \\
		&\sup_{s \in[0,1] } | \wt Z_{2,m,n}(s) - g_{1,2}(s)|  \to 0. \label{E:WeakConvergenceTwoSample13}
	\end{align}
	Indeed, \eqref{E:WeakConvergenceTwoSample12} is bounded above as follows
	\begin{align*}
		&\sup_{s\in [0,1]} |\wt Z_{1,m,n}(s) -  G_{1,1}(H_m,J_n)(s)| \\
		&\lesssim \sup_{s \in[0,1] } s^{1-\gamma} \  \Big | \frac{ s m / \floor{ s m} }{ \sqrt{ m/(m+n)}} - \frac{1}{\sqrt{\tau}} \Big| \ | H_m(s)| \\
		&\quad + \sup_{s \in[0,1] } s^{1-\gamma} \ \Big | \frac{ s n / \floor{ s n} }{ \sqrt{ n/(m+n)}} - \frac{1}{\sqrt{1-\tau}} \Big| \ | J_n(s)| \rightarrow 0
	\end{align*}
	in probability. And similarly for \eqref{E:WeakConvergenceTwoSample13}
	\begin{align*}
		\sup_{s \in[0,1] } | \wt Z_{2,m,n}(s) - g_{1,2}(s)|
		&\lesssim \sup_{s \in[0,1] } s^\gamma \  \frac{\sqrt{m}}{\floor{ms}} | H_m(s) | 
		+ \sup_{s \in[0,1] } s^\gamma \  \frac{\sqrt{n}}{\floor{ns}} | J_n(s) | \rightarrow 0
	\end{align*}
	in probability by Lemma~\ref{L:MaxConvProp} -- using the assumed strong law of large numbers for $m^{-1} \ol X_m(1)$ and $n^{-1} \ol Y_n(1)$.
Hence by Slutsky's theorem the first term in \eqref{E:WeakConvergenceTwoSample-7} converges to zero in probability.

The second term in \eqref{E:WeakConvergenceTwoSample-7} can be bounded above by choosing the identity in the definition of the Skorokhod metric, i.e.,
\begin{align*}
	&d_1(G_1(H_m,J_n) - (\cdot)^2 Z_{m,n},~ G_1(H_m,J_n)-(\cdot)^2 F_1(H_m,J_n)  )\\
	&\leq \sup_s |s^2 Z_{m,n} - s^2 F_1(H_m,J_n)| = |Z_{m,n} - F_1(H_m,J_n)|.
\end{align*}
Then once more, relying on the results of step two, the right-hand side tends to zero in probability. Thus, $\eqref{E:WeakConvergenceTwoSample-3}$ is established 
and together with \eqref{E:WeakConvergenceTwoSample-8} we conclude by Slutsky's theorem 
\begin{align*}
	\wt Z_{m,n} - (\cdot)^2 Z_{m,n} \Rightarrow F_2(\sqrt{\Gamma_X} B_1, \sqrt{\Gamma_Y} B_2).
\end{align*}
Also, it is straightforward to see 
\begin{align*}
	&F_2(\sqrt{\Gamma_X} B_1, \sqrt{\Gamma_Y} B_2)(s)\\
	&\stackrel{\cL}{=} \left( 2s (\E[X_0] - \E[Y_0]) \sqrt{\frac{\Gamma_X}{\tau} + \frac{\Gamma_Y}{1-\tau}} \big( B(s) - sB(1) \big) \right)^2 = \xi^2 s^2 \big( B(s) - sB(1) \big)^2.
\end{align*}
Finally, since integration is a continuous functional on $D_1$ (w.r.t.\ $d_1$), we can deduce on the one hand
\begin{align*}
	\left\{ \int_0^1 F_2(H_m,J_n)(s)^2 \nu(\diff s) \right\}^{1/2} \Rightarrow \xi \left\{ \int_0^1 s^2 \big(B(s) - s B(1)\big)^2 \, \nu(\diff s) \right\}^{1/2}.
\end{align*}
And on the other hand,
\begin{align}
	\sqrt{m+n}\wt V_{m,n} - \left\{ \int_0^1 F_2(H_m,J_n)(s)^2 \nu(\diff s) \right\}^{1/2} \to 0 \label{E:WeakConvergenceTwoSample14}
\end{align}
in probability.

{\it Step 4.} It is just a formality to remark that the weak convergence which we established in steps two and tree carries over to the two dimensional setting: Indeed,
\begin{align*}
	&\Bigg( F_1(H_m, J_n), \left\{ \int_0^1 F_2(H_m,J_n)(s)^2 \nu(\diff s) \right\}^{1/2} \Bigg) \\
	&\quad \Rightarrow (F_1(\sqrt{\Gamma_X} B_1, \sqrt{\Gamma_Y} B_2), \left\{ \int_0^1 F_2(\sqrt{\Gamma_X} B_1, \sqrt{\Gamma_Y} B_2)(s)^2 \nu(\diff s) \right\}^{1/2} \\
	&\qquad \overset{\cL}{=} \Bigg( \xi B(1),\xi \left\{ \int_0^1 s^2 \big(B(s) - s B(1)\big)^2 \, \nu(\diff s) \right\}^{1/2} \Bigg)
\end{align*}
and the convergence results in \eqref{E:WeakConvergenceTwoSample15}  and \eqref{E:WeakConvergenceTwoSample14} remain true as well when considered jointly in a vector. This completes the proof.
\end{proof}

\begin{lemma}\label{L:MaxConvProp}
Let $(Z_n)_n$ be a sequence of identically distributed, integrable random variables and assume that the corresponding partial sum $\ol{Z}_k = \sum_{i=1}^k Z_i$ satisfies 
a SLLN, i.e., $k^{-1}\ol{Z}_k \to \E[Z_0]$ almost surely.
Then for all $\delta$, $\epsilon>0$ and $\gamma \in (0,1)$ there is a constant $M\in \N$ such that
\begin{align}\label{E:MaxConvProp2}
	\p\left( \max_{k\le m } \Big( \frac{k+1}{m}\Big)^\gamma \big| k^{-1} \ol Z_k - \E[Z_0]\big| > \epsilon \right) < \delta \qquad \forall\, m\ge  M.
\end{align}
\end{lemma}

\begin{proof}[Proof of Lemma~\ref{L:MaxConvProp}]
By assumption there is a measurable set $\Omega_0 \subseteq \Omega$ with $\p(\Omega_0) = 1$ and 
\begin{align*}
	\big|k^{-1} \ol Z_k(\omega) - \E[Z_0]\big| \to 0 \qquad \forall\, \omega \in \Omega_0.
\end{align*}
Hence, for each $\omega \in \Omega_0$, there exists a $M_0(\omega) \in \N$ such that 
\begin{align}\label{E:MaxConvProp6}
	\big|k^{-1} \ol Z_k(\omega) - \E[Z_0]\big| < \frac{\epsilon}{2^{\gamma}} \qquad \forall\, k \geq M_0(\omega).
\end{align}	
For $\omega \in \Omega\backslash \Omega_0$ set $M_0(\omega) = 0$. Then  $M_0 \colon \Omega \to \N\cup \{\infty\}$ is almost surely bounded and therefore
we can find a constant $M_1\in\N$ such that $\p(M_0 > M_1) \leq \delta/2$. Next, observe that for $\alpha\in(0,1)$ the probability in \eqref{E:MaxConvProp2} can be bounded 
as follows:
\begin{align}
	&\p\left( \max_{k\le m } \Big( \frac{k+1}{m}\Big)^\gamma \big| k^{-1} \ol Z_k - \E[Z_0]\big| > \epsilon \right) \nonumber\\
	&\leq \p\left( \max_{k\le m } \Big( \frac{k+1}{m}\Big)^\gamma \big| k^{-1} \ol Z_k - \E[Z_0]\big| > \epsilon,~ M_0 \leq M_1 \right) + \p(M_0 > M_1) \nonumber\\
	&\leq \p\left( \max_{k\le \floor{m^\alpha} } \Big( \frac{k+1}{m}\Big)^\gamma \big| k^{-1} \ol Z_k - \E[Z_0]\big| > \epsilon,~ M_0\leq M_1\right) \nonumber\\
	&\quad+ \p\left(\max_{\floor{m^\alpha} < k\leq m} \Big( \frac{k+1}{m}\Big)^\gamma \big| k^{-1} \ol Z_k - \E[Z_0]\big| > \epsilon,~ M_0 \leq M_1 \right)
	 + \frac{\delta}{2} \nonumber\\
	&\leq  \p\left( \Big( \frac{\floor{m^\alpha}+1}{m}\Big)^\gamma \max_{k\le \floor{m^\alpha} } \big| k^{-1} \ol Z_k - \E[Z_0]\big| > \epsilon\right) \nonumber\\
	&\quad+ \p\left(\max_{\floor{m^\alpha} < k\leq m} \big| k^{-1} \ol Z_k - \E[Z_0]\big| > \frac{\epsilon}{2^{\gamma}},~ M_0 \leq M_1 \right) 
	+ \frac{\delta}{2} \label{E:MaxConvProp4}
\end{align} 
Furthermore, applying Markov's inequality to the first term in \eqref{E:MaxConvProp4} shows 
\begin{align}
	&\p\left( \Big( \frac{\floor{m^\alpha}+1}{m}\Big)^\gamma \max_{k\le \floor{m^\alpha} } \big| k^{-1} \ol Z_k - \E[Z_0]\big| > \epsilon\right) \nonumber\\
	&\leq \frac{1}{\epsilon} \Big( \frac{\floor{m^\alpha}+1}{m}\Big)^\gamma\, \E\left[ \max_{k\le \floor{m^\alpha} } \big| k^{-1} \ol Z_k - \E[Z_0]\big|  \right] \nonumber\\
	&\leq \frac{1}{\epsilon} \Big( m^{\alpha-1}+\frac{1}{m}\Big)^\gamma m^\alpha \, \E \big[ |Z_1 - \E[Z_0]| \big] \nonumber\\
	&\leq \frac{1}{\varepsilon} \left( m^{\alpha(\gamma+1)-\gamma} + m^{\alpha-\gamma} \right) \E \big[ |Z_1 - \E[Z_0]| \big]. \label{E:MaxConvProp5}
\end{align}
Hence, if we let $0<\alpha<\gamma/(1+\gamma)$ both exponents of $m$ become negative and we can find $M_2 \in \N$ such that \eqref{E:MaxConvProp5} is less than 
$\delta/2$ for all $m\geq M_2$. Finally, we put $M := M_1^{1/\alpha} \vee M_2$. Then $m \geq M$ implies $m^{\alpha} \geq M_1$ and 
the second term in \eqref{E:MaxConvProp4} vanishes therefore for all $m \geq M$ due to \eqref{E:MaxConvProp6}, which completes the proof.
\end{proof}

\begin{cor}\label{Cor:SLLNforL^2-m-processes}
Any stationary, $L^2$-$m$-approximable sequence $(Z_n)_n$ satisfies the strong law of large numbers, that is $n^{-1}(Z_1+\ldots+Z_n) \to \E[Z_0]$ almost surely.
\end{cor}
\begin{proof}
Without loss of generality we can assume $\E[Z_0] = 0$ by centering. First observe that $Z_0$ and $Z_j^{(j)}$ are independent and therefore
\begin{align*}
	\sum_{j\in\N} | \E[ Z_0 Z_j ]| &\le \sum_{j\in\N} | \E[ Z_0 (Z_j - Z_j^{(j)}) ]| + \sum_{j\in\N} | \E[ Z_0  Z_j^{(j)} ]| \\
	&\le \E[Z_0^2]^{1/2} \sum_j \E[ (Z_0 - Z_0^{(j)})^2]^{1/2} < \infty.
\end{align*}
In particular for any $I\subseteq \N$, $\V( \sum_{j\in I} Z_j ) \le C \# I$ for a constant $C\in\R_+$ which does not depend on $I$.
Consequently, for $S_n = \sum_{i=1}^n Z_i$, we can rely on a classical proof technique for the strong law of large numbers and use the Borel-Cantelli lemma to see that both
\begin{align*}
	\frac{1}{n^2}  S_{n^2} \rightarrow 0 \quad a.s.
\text{ and }	\max_{n^2+1 \le k \le (n+1)^2 } \  \frac{1}{n^2} (S_k- S_{n^2}) \rightarrow 0 \quad a.s.
\end{align*}
In particular, $\max_{n^2+1 \le k \le (n+1)^2 } \frac{1}{k} S_k \to 0$ for $n\to \infty$ with probability $1$. Clearly, this implies the strong law of large numbers.
\end{proof}

\subsection{Two sample tests for inco-variances}

\begin{lemma}[Cohen-Steiner et al. \cite{cohen2010lipschitz}]\label{L:CohenSteiner} Let $1\le p'\le r < \infty$ and let $\pd(\cX),\pd(\cY)$ be persistence diagrams whose $p'$-total persistence is bounded from above. Then
\begin{align*}
	W_r(D,E) &\le \Big( \frac{ \pers_{p'}(\pd(\cX))	+ \pers_{p'}(\pd(\cY)) }{2} \Big)^{1/r} W_\infty(\pd(\cX),\pd(\cY))^{1- p'/r} \\
	&\le \Big( \frac{ \pers_{p'}(D)	+ \pers_{p'}(E) }{2} \Big)^{1/r} d_H(\cX,\cY)^{1- p'/r}.
\end{align*}
Here the $p'$-total persistence is defined as $\pers_{p'}(D) = \sum_{(d_x,b_x)\in D} (d_x-b_x)^{p'}$ .
\end{lemma}

\begin{lemma}[Kusano et al. \cite{kusano2017kernel}[Corollary 5]\label{L:WassersteinBound}
Let $\fD$ be a compact triangulable subset of $\R^d$. Let $\cX$ and $\cY$ be finite subsets of $\fD$ and let $r\ge p' >d$. Then we have for the persistence diagrams $\pd(\cX)$ and $\pd(\cY)$ from the \v Cech filtration (for any dimension)
\begin{align*}
	W_r( \pd(\cX), \pd(\cY) ) &\le \Big(	\frac{p'}{p'-d} \ C_{\fD} \ (\diam \fD)^{p'-d} \Big)^{1/r} W_\infty( \pd(\cX), \pd(\cY) )^{1-p'/r} \\
	&\le \Big(	\frac{p'}{p'-d} \ C_{\fD} \ (\diam \fD)^{p'-d} \Big)^{1/r} d_H(\cX,\cY)^{1-p'/r}
\end{align*}
for a constant $C_{\fD}\in\R_+$ which only depends on $\fD$.
\end{lemma}

\begin{lemma}\label{Thrm:PropertiesKernelIncoVar}\
\begin{itemize}
\item [(A)] Let $d<r$. 
\begin{itemize}
	\item [(1)] The kernel \eqref{E:KernelU-StatisticForIncoVar} is $\rho$-Hölder continuous w.r.t.\ the product Hausdorff metric on $\dN \times \dN$ for each $\rho \in (0,1-d/r)$
\begin{align*}
	\big| h(\cX,\cY) - h(\cX',\cY') \big| 
	\leq L_{\rho} ~\big(d_H(\cX,\cX') + d_H(\cY,\cY')\big)^\rho,
\end{align*}
where the Hölder constant satisfies
$$
	L_\rho = 4 \Big(	\frac{r}{r-d} \  \frac{r(1-\rho)}{r(1-\rho)-d} \ C_{\fD}^2 \  (\diam \fD)^{r(2-\rho)-2d} \Big)^{1/r}.
$$
	\item [(2)] $h$ satisfies for any $\cU,\cV \in \dN |_\fD$ (i.e., point processes which are restricted to $\fD$)
$$
	h(\cU,\cV) \le  \Big(	\frac{r}{r-d} \ C_{\fD} \ (\diam \fD)^{r-d} \Big)^{1/r} =: C^*.
$$
\end{itemize}
\item [(B)] Let $\cX,\cY,\cX',\cY$ be finite with cardinality at most $T$.
	\begin{itemize}
		\item [(1)] $h$ satisfies for any $\cX,\cY \in \dN |_\fD$ with cardinality at most $T$
$$
	h(\cX,\cY) \le  \Big( (T^{k+1} + T^{k+2})	 \ (\diam \fD)^{r} \Big)^{1/r} =: C_*
$$
	if the underlying dimension of the persistence diagram is $k$.
	\item [(2)] The kernel \eqref{E:KernelU-StatisticForIncoVar} is $1-1/r$-Hölder continuous w.r.t.\ the product Hausdorff metric on $\dN \times \dN$. 
	\end{itemize}
\end{itemize}
\end{lemma}

\begin{proof}
{\it Part (A).} We begin with statement (ii) which is a direct consequence of Lemma~\ref{L:WassersteinBound} in the special case $p'=r$. Note that both $\cU|_{\fD}$ and $\cV|_{\fD}$ are finite with probability 1 because $\fD$ is compact and elements of $\dN$ are locally finite.

For statement (i), we apply the reverse triangle inequality
\begin{align}
	|h(\cX,\cY) - h(\cX',\cY')| &= | W_r^2( \pd(\cX),\pd(\cY) ) - W_r^2( \pd(\cX'),\pd(\cY') )| \\
	\begin{split}\label{E:PropertiesKernelIncoVar0}
	&\le \Big\{ W_r( \pd(\cX),\pd(\cY) ) +  W_r( \pd(\cX'),\pd(\cY') ) \Big\} \\
	&\qquad\qquad \cdot  \Big\{ W_r( \pd(\cX),\pd(\cX') ) +  W_r( \pd(\cY),\pd(\cY') ) \Big\}.
	\end{split}
\end{align}
Again using Lemma~\ref{L:WassersteinBound}, the first factor on the RHS is uniformly bounded above as follows
\begin{align*}
	 &W_r( \pd(\cX),\pd(\cY) ) +  W_r( \pd(\cX'),\pd(\cY') ) \le 2 C^*.
\end{align*}
Moreover, for each $p'\in (d,r)$ the second factor is at most
\begin{align*}
	&\Big(	\frac{p'}{p'-d} \ C_{\fD} \ (\diam \fD)^{p'-d} \Big)^{1/r} \Big\{ d_H(\cX,\cX')^{1-p'/r} + d_H(\cY,\cY')^{1-p'/r} \Big\} \\
	&\le 2\Big(	\frac{p'}{p'-d} \ C_{\fD} \ (\diam \fD)^{p'-d} \Big)^{1/r} \Big\{ d_H(\cX,\cX') + d_H(\cY,\cY')\Big\}^{1-p'/r}.
\end{align*}
Hence, the kernel $h$ is Hölder continuous for each $\rho \in (0,1-d/r)$ with Hölder constant
$$
		4 C^* \Big( \frac{r(1-\rho)}{r(1-\rho)-d} C_\fD (\diam \fD)^{r(1-\rho)-d} \Big)^{1/r}.
$$
{\it Part (B).} We rely this time on Lemma~\ref{L:CohenSteiner} and the trivial upper bound $\pers_{p'}(\pd(\cX)) \le (T^{q+1}+T^{q+2})(\diam \fD)^{p'}$, the calculations are very similar and we skip them for sake of brevity.
\end{proof}

\begin{lemma}\label{L:ConvergenceSigmas}
Given the assumptions of Theorem~\ref{Thrm:IncoVarianceJointConvergence}, we have
$$
  \E[|\wh\sigma_{\cX}^2(s,t) - st \sigma_{\cX}^2|^2]\to 0, \quad m\to\infty \text{ and } \E[| \wh\sigma_{\cY}^2(s,t)- st\sigma_{\cY}^2)|^2], \quad n\to\infty
$$
for all $s,t\in[0,1]$.
\end{lemma}
\begin{proof}
We make use of Corollary~\ref{C:VarUh}.
\begin{align*}
	&\E[ |\wh\sigma_{\cX}^2(s,t) - st \sigma_{\cX}^2|^2 ] \\
	&= \frac{1}{4} \ \E\Big[ \Big( \frac{1}{m(m-1)} \sum_{i=1}^{\lf ms \rf} \sum_{j=1\atop j\not= i}^{\lf mt \rf}  W_r^2(\pd(\cX_i), \pd(\cX_j)) - st \E[ W_r^2(\pd(\cX_0), \pd(\cX')) ] \Big)^2 \Big] \\
	&\lesssim n^{-1}
\end{align*}
The result holds in the same spirit for $\wh\sigma_{\cY}^2(s,t)$.
\end{proof}

\begin{proof}[Proof of Theorem~\ref{Thrm:IncoVarianceJointConvergence}]
Again the proof splits in four major steps and is very similar to the proof of Theorem~\ref{Thrm:WeakConvergenceTwoSampleRelevantDifference_General}.

{\it Step 1.}
In the first step we verify that the conditions of Assumption~\ref{AssumptionFCLT}~(p,q,$\rho$) are satisfied in case (a) and case (b).

{\it Case (a).} If $\# \cX_i, \#\cY_j$ are uniformly bounded with probability 1 for all $i,j\in\Z$ the kernel in \eqref{E:KernelU-StatisticForIncoVar} satisfies part (A1) in Assumption~\ref{AssumptionFCLT}~(p,q,$\rho$) for $\rho=1-1/r$ and any $q\ge 1$ by Lemma~\ref{Thrm:PropertiesKernelIncoVar} part (B).

{\it Case (b).} The kernel satisfies part (A1) in Assumption~\ref{AssumptionFCLT}~(p,q,$\rho$)  for any $q\geq 1$ and each $\rho < 1-d/r$ by Lemma~\ref{Thrm:PropertiesKernelIncoVar} part (A).

Thus, if we choose $p=1$ and $q=\infty$ in both cases (a) and (b), the data $(\cX_i)_i$ and $(\cY_i)_i$ satisfy the approximation condition (A2) in Assumption \ref{AssumptionFCLT}~(p,q,$\rho$) because we require 
$$
	\sum_{m \ge 1} m~ \E[ d_H( \cX_0, \cX_0^{(m)} )^{\rho} ] < \infty \text{ and } \sum_{m \ge 1} m ~ \E[ d_H( \cY_0, \cY_0^{(m)} )^{\rho} ] < \infty 
$$
for $\rho=1-1/r$ in case (a) and in case (b) for some $\rho\in (0,1-d/r)$ where additionally $d<r$.
Hence, we can apply Theorem \ref{Thrm:FCLTforUStatistic} to the processes \eqref{E:IncoVarProcess} and obtain due to the independence between the samples $(\cX_i)_i$ and $(\cY_i)_i$
\begin{align} \label{E:IncoVarianceJointConvergence1}
		\begin{pmatrix}
			\sqrt{m}(\wh\sigma_\cX^2 (s,t) - st \sigma_\cX^2) : 0\le s,t\le 1 \\
			\sqrt{n}(\wh\sigma_\cY^2 (s,t) - st \sigma_\cY^2): 0\le s,t\le 1
		\end{pmatrix}
		&\Rightarrow
		\begin{pmatrix}
			t \cW_{\Gamma_\cX}(s) + s \cW_{\Gamma_\cX}(t)  : 0\le s,t\le 1 \\
			t \cW_{\Gamma_\cY}(t) + s \cW_{\Gamma_\cY}(s)  : 0\le s,t\le 1
		\end{pmatrix}
\end{align}
in the product space $D_2\times D_2$, where $D_2=D([0,1]^2)$. Note that the Wiener processes on the right-hand side $\cW_{\Gamma_\cX}, \cW_{\Gamma_\cY}$ are necessarily independent. Thus, we can characterize the Wiener processes as $\cL(\cW_{\Gamma_\cX}) = \cL(\sqrt{\Gamma_\cX} B_1)$ and $\cL(\cW_{\Gamma_\cY}) = \cL(\sqrt{\Gamma_\cY} B_2)$, where $B_1$ and $B_2$ denote two independent Brownian motions on $[0,1]$.

{\it Step 2.}
Now, we can consider the joint behavior of two dimensional random vector given on the left-hand side of \eqref{E:IncoVarianceJointConvergence0}. The first component is
\begin{align}
	&\sqrt{m+n}( \wh{D}_{m,n}^2 - D^2) = \sqrt{m+n}( \wh{D}_{m,n} - D)( \wh{D}_{m,n} + D) \nonumber\\
	\begin{split}\label{E:IncoVarianceJointConvergence2}
	&= \left(\sqrt{\frac{m+n}{m}} \sqrt{m} (\wh\sigma_{\cX}^2(1,1) - \sigma_{\cX}^2) 
	- \sqrt{\frac{m+n}{n}} \sqrt{n} (\wh\sigma_{\cY}^2(1,1) - \sigma_{\cY}^2) \right)\\
	&\quad\qquad \times \left( \wh\sigma_{\cX}^2(1,1) - \wh\sigma_{\cY}^2(1,1) + \sigma_{\cX}^2 - \sigma_{\cY}^2  \right) =: Z_{m,n}. 
	\end{split}
\end{align}	
The second component is an integral and the square root of its integrand is
\begin{align}
	&\sqrt{m+n}(\wh{D}_{m,n}^2(s,t) \mp ( stD )^2 - (st\wh{D}_{m,n})^2) \nonumber\\
	&=\sqrt{m+n} ( \wh{D}_{m,n}(s,t) - stD )( \wh{D}_{m,n}(s,t) + stD ) \nonumber\\ 
	&\quad - (st)^2\sqrt{m+n}(\wh{D}_{m,n}- D)(\wh{D}_{m,n}+D) \nonumber\\ 
	\begin{split}\label{E:IncoVarianceJointConvergence3}
	&= \left(\sqrt{\frac{m+n}{m}} \sqrt{m} (\wh\sigma_{\cX}^2(s,t) - st\sigma_{\cX}^2) 
	- \sqrt{\frac{m+n}{n}} \sqrt{n} (\wh\sigma_{\cY}^2(s,t) - st\sigma_{\cY}^2)\right)\\
	&\quad\qquad \times ( \wh\sigma_{\cX}^2(s,t) + st\sigma_{\cX}^2 - \wh\sigma_{\cY}^2(s,t) - st\sigma_{\cY}^2  )\\
	&\quad\qquad - (st)^2 \sqrt{m+n}(\wh{D}_{m,n}- D)(\wh{D}_{m,n}+D) =: \wt Z_{m,n}(s,t).
\end{split}
\end{align}
Note that the last term on the right-hand side of \eqref{E:IncoVarianceJointConvergence3} is the $(st)^2$-multiple of the left-hand side of \eqref{E:IncoVarianceJointConvergence2}.

{\it Step 3.}
Next, we prove the desired convergence, which essentially is an application of the continuous mapping theorem. To this end, we define appropriate mappings $F_1, F_2$ on $D_2\times D_2$.

We begin with the first coordinate which is given in \eqref{E:IncoVarianceJointConvergence2} and define the functional
$$
	F_1 \colon  D_2\times D_2 \to \R, \quad (u,v)\mapsto \Big(\frac{u(1,1)}{\sqrt{\tau}}  - \frac{v(1,1)}{\sqrt{1-\tau}}  \Big) (2\sigma_{\cX}^2-2\sigma_{\cY}^2).
$$
Given two continuous functions $u,v$ in $D_2$, $F_1$ is continuous at $(u,v)$ because the point evaluation is a continuous functional at  continuous functions in the Skorokhod space.

Moreover, $m/(m+n)\to\tau$ as well as $\wh\sigma_\cX^2\to \sigma_\cX^2$ and $\wh\sigma_\cY^2\to \sigma_\cY^2$ in probability by Lemma~\ref{L:ConvergenceSigmas}. And all in all, it follows with some calculations that
\begin{align}\label{E:IncoVarianceJointConvergence3b}
	Z_{m,n} - F_1(\sqrt{m}(\wh\sigma_\cX^2(1,1)-\sigma^2_\cX),\sqrt{n}(\wh\sigma_\cY^2(1,1)-\sigma^2_\cY) ) \to 0 
\end{align}
in probability.

Moreover, as the limits on the right-hand side of \eqref{E:IncoVarianceJointConvergence1} are continuous elements of $D_2$, we can apply the continuous mapping theorem to $F_1$
\begin{align*}
	&F_1(\sqrt{m}(\wh\sigma_\cX^2(1,1)-\sigma^2_\cX),\sqrt{n}(\wh\sigma_\cY^2(1,1)-\sigma^2_\cY) ) \\
	 &\quad \Rightarrow F_1\big( (t W_{\Gamma_\cX}(s)+s W_{\Gamma_\cX}(t))_{s,t}, (t W_{\Gamma_\cY}(s)+s W_{\Gamma_\cY}(t))_{s,t} \big) \nonumber \\
	&\qquad =  \left(\frac{2}{\sqrt{\tau}} W_{\Gamma_X}(1) 
	- \frac{2}{\sqrt{1-\tau}} W_{\Gamma_Y}(1)\right)(2\sigma_{\cX}^2 - 2\sigma_{\cY}^2 ).
\end{align*}
Next, we consider $\wt Z_{m,n}$ which is given in \eqref{E:IncoVarianceJointConvergence3}. Define the operator
\begin{align*}
	&F_2 \colon D_2\times D_2 \to D_2,\\
	&\qquad \quad (u,v)\mapsto \Big( 2st ~ \Big\{ \frac{u(s,t)}{\sqrt \tau} - \frac{ v(s,t)}{\sqrt{1-\tau}}\Big\} (\sigma_\cX^2 - \sigma_\cY^2 ) - (st)^2 F_1(u,v) \Big)_{s,t}.
\end{align*}
Now, it follows with some calculations that
\begin{align}
\begin{split}\label{E:IncoVarianceJointConvergence3c}
	d_2\Big( \wt Z_{m,n}, F_2\big( &[ 
			\sqrt{m}(\wh\sigma_\cX^2 (s,t) - st \sigma_\cX^2)]_{s,t},\\
			&	[\sqrt{n}(\wh\sigma_\cY^2 (s,t) - st \sigma_\cY^2)]_{s,t}
			\big) \Big) \to 0 \text{ in probability.}
\end{split}\end{align}
Observe that $F_2$ is discontinuous only on $\mathrm{Dis}(F_2) = D_2\backslash\cC_2 \times D_2\backslash \cC_2$, where $\cC_2$ denotes the set of continuous functions on $[0,1]^2$.
Now, 
$$	
	\p\big( ( [ tW_{\Gamma_X}(s) + sW_{\Gamma_X}(t) ]_{s,t}, [ tW_{\Gamma_Y}(s) + sW_{\Gamma_Y}(t) ]_{s,t} ) \in \mathrm{Dis}(F_2) \big) = 0.
$$
Consequently, an application of the continuous mapping theorem yields
\begin{align*}
	 &F_2\big( [	\sqrt{m}(\wh\sigma_\cX^2 (s,t) - st \sigma_\cX^2)]_{s,t},	[\sqrt{n}(\wh\sigma_\cY^2 (s,t) - st \sigma_\cY^2)]_{s,t} \big) \\
			  &\Rightarrow F_2\big( [ tW_{\Gamma_X}(s) + sW_{\Gamma_X}(t) ]_{s,t}, [ tW_{\Gamma_Y}(s) + sW_{\Gamma_Y}(t) ]_{s,t} \big) \\
	 &= \Bigg( 2st \left\{  \frac{ tW_{\Gamma_X}(s) + sW_{\Gamma_X}(t) }{\sqrt{\tau}} 
	- \frac{tW_{\Gamma_Y}(s) + sW_{\Gamma_Y}(t) }{\sqrt{1-\tau}}  \right\} (\sigma_\cX^2 - \sigma_\cY^2 ) \\
	&\quad\qquad - 4(st)^2 \left(\frac{W_{\Gamma_X}(1) }{\sqrt{\tau}} 
	- \frac{ W_{\Gamma_Y}(1)}{\sqrt{1-\tau}}\right)(\sigma_{\cX}^2 - \sigma_{\cY}^2 ). \Bigg)_{s,t}
\end{align*}
in $D_2$. 
The law of the last right-hand side equals
\begin{align*}
	&\Bigg( \left\{ \sqrt{\frac{\Gamma_X}{\tau}} (tB_1(s) + sB_1(t)) - \sqrt{\frac{\Gamma_Y}{1-\tau}} (tB_2(s) + sB_2(t)) \right\}  \times 2st(\sigma_\cX^2 - \sigma_\cY^2 ) - 2(st)^2\xi B(1) \Bigg)_{s,t} \\
	&\overset{\cL}{=}~ \Big( \xi \Big(st\big(tB_1(s) + sB_1(t) - 2stB_1(1)\big) \Big) \Big)_{s,t}.
\end{align*}
Moreover, set
$$
	G\colon D_2 \to\R, \quad u\mapsto \left\{ \int_{[0,1]^2} u(s,t)^2 \nu(\diff( s, t)) \right\}^{1/2}.
$$
Since taking the pointwise square, integration and the square root are $\p_{(B_1,B_2)}$-a.s. continuous functionals on $D_2$, the composition $ G\circ F_2$ is an $\p_{(B_1,B_2)}$-a.s. continuous function. Thus, using the result from \eqref{E:IncoVarianceJointConvergence3c}, we have on the one hand 
\begin{align}
	\begin{split}\label{E:IncoVarianceJointConvergence3d}
	&\sqrt{m+n}\wh V_{m,n} - G\circ F_2\big( [	\sqrt{m}(\wh\sigma_\cX^2 (s,t) - st \sigma_\cX^2)]_{s,t},	[\sqrt{n}(\wh\sigma_\cY^2 (s,t) - st \sigma_\cY^2)]_{s,t} \big) \\
	&= G( \wt Z_{m,n}) -  G\circ F_2\big( [	\sqrt{m}(\wh\sigma_\cX^2 (s,t) - st \sigma_\cX^2)]_{s,t},	[\sqrt{n}(\wh\sigma_\cY^2 (s,t) - st \sigma_\cY^2)]_{s,t} \big) \to 0
	\end{split}
\end{align}
in probability.

And on the other hand, an application of the continuous mapping theorem yields 
\begin{align*}
	&G\circ F_2\big( [	\sqrt{m}(\wh\sigma_\cX^2 (s,t) - st \sigma_\cX^2)]_{s,t},	[\sqrt{n}(\wh\sigma_\cY^2 (s,t) - st \sigma_\cY^2)]_{s,t} \big) \\
	&\quad \Rightarrow G\Big(  F_2\big( [ tW_{\Gamma_X}(s) + sW_{\Gamma_X}(t) ]_{s,t}, [ tW_{\Gamma_Y}(s) + sW_{\Gamma_Y}(t) ]_{s,t} \big) \Big) \\
	&\qquad \overset{\cL}{=}
	\left\{ \xi^2 \int_{[0,1]^2} \Big[ st\big(tB_1(s) + sB_1(t) - 2st B_1(1)\big) \Big]^2  \, \nu(\diff( s, t)) \right\}^{1/2}.
\end{align*}
{\it Step 4.} Finally, the map 
$$
	D_2\times D_2 \to \R^2, \quad (u,v) \mapsto
	\begin{pmatrix}
		F_1(u,v)\\
		G\circ F_2(u,v)
	\end{pmatrix}
$$
is continuous w.r.t.\ the product metrics and the convergence in probability shown in \eqref{E:IncoVarianceJointConvergence3b} and \eqref{E:IncoVarianceJointConvergence3d} carries over to the two-dimensional setting as well. This completes the proof.
\end{proof}

\subsection{Functional CLT for U-statistics}\label{Subsection_DetailsUStats}

\begin{proof}[Proof of Theorem~\ref{T:2ParaDonsker}]
Let $g\colon (D_1, \fD_1) \to (D_2, \fD_2)$ be defined by $g(x)(s,t) := tx(s) + sx(t)$, $s,t\in[0,1]$. Consider the process 
	\begin{align*}
		X_n \colon [0,1]\to\R,\quad  t\mapsto \frac{1}{\sqrt{n}} \sum_{i=1}^{ \floor{n t}} Z_i	
	\end{align*}
Then $Y_n= g(X_n)$ and $ X_n \Rightarrow \cW_{\Gamma}$ in $(D_1,\fD_1)$ by the functional central limit theorem (Theorem A.1) given in Aue et al. \cite{aue2009break}.
Thus it remains to verify that $g\colon D_1 \to D_2$ is continuous with respect to the Skorokhod metrics ($q=1,2$)
\begin{align}\label{SkorokhodMetric}
	d_q(x,y) := \inf_{\lambda \in \Lambda_q} \left\{ \max\Big\{ \sup_t |x(t) - y(\lambda(t))|, \sup_t \|\lambda(t) -t\| \Big\} \right\}, \quad x,y\in D_q,
\end{align}
where the infimum is taken over the class $\Lambda_q$ of all ``time transformations'' $\lambda(t)=(\lambda_1(t_1),\ldots, \lambda_q(t_q))$ such that
$\lambda_r\colon[0,1] \to [0,1]$ is continuous and strictly increasing with $\lambda_r(0)=0$ and $\lambda_r(1)=1$.\\
Now, assume that $x,y \in D_1$ with $d_1(x,y) < \delta$. Then there is a $\lambda^* \in \Lambda_1$ such that 
\begin{align*}
	\sup_t |x(t) - y(\lambda^*(t))| < \delta \quad\text{and}\quad \sup_t |\lambda^*(t) -t| < \delta. 
\end{align*}
Put $\wt\lambda = (\lambda^*,\lambda^*) \in \Lambda_2$, then 
\begin{align*}
	&| g(x)(s,t) - g(y)(\wt\lambda(s,t)) |\\
	=& |tx(s) \mp \lambda^*(t)x(s) -\lambda^*(t)y(\lambda^*(s)) + sx(t) \mp \lambda^*(s)x(t) - \lambda^*(s)y(\lambda^*(t))|\\
	\leq& |t-\lambda^*(t)| |x(s)| + |\lambda^*(t)||x(s) - y(\lambda^*(s))| + |s-\lambda^*(s)| |x(t)|\\
	&+ |\lambda^*(s)||x(t) - y(\lambda^*(t))|\\
	\leq& 2(1+ \sup |x|)\delta
\end{align*}
as well as 
\begin{align*}
	\|\wt\lambda(s,t) - (s,t)\| = |\lambda^*(s) - s| + |\lambda^*(t) - t| < 2\delta. 
\end{align*}
Now, let $\epsilon>0$ and $x\in D_1$ be arbitrary but fixed. Set $\delta \coloneqq \epsilon/ [ 2 (1+ \sup |x|)]$.
Then we may conclude for all $y\in D_1$ with $d_1(x,y) < \delta$, that
\begin{align*}
	d_2(g(x),g(y)) < 2(1+ \sup |x|)\delta = \epsilon.
\end{align*}
This shows the continuity of $g$ and the assertion immediately follows from the continuous mapping theorem, i.e., 
$Y_n = g(X_n)$ converges weakly to $g(X) = Y$. 
\end{proof}

\begin{proof}[Proof of Theorem~\ref{T:MomentCondition}]
The proof is divided into several steps.\\[5pt]

{\it (1) Reduction to equidistant grid.} 
Firstly, we argue that it is sufficient to consider blocks $A= (s, u]\times (t, v]$ with corner points $(s,t), (u,v)$ in the
equidistant grid $\Pi_n$ of $[0,1]^2$, where the grid is given by
$$	
		\Pi_n = \Big\{ \Big(\frac{i}{n}, \frac{j}{n} \Big): 0\le i,j \le n \Big\}.
$$
Let $(s,t)\in [0,1]^2$ and define $\ul s = \floor{ns}/n$ and $\ul t = \floor{nt}/n$. Then $(\ul s,\ul t)$ is the point on the grid 
$\Pi_n$ which satisfies $U_n(h_2)(s,t) = U_n(h_2)(\ul s, \ul t)$. In particular, setting $\ul A = (\ul s, \ul u]\times (\ul t, \ul v]$ for 
$A=(s, u]\times (t, v]$ yields immediately $U_n(h_2)(A) = U_n(h_2)(\ul A)$.
So, in the following, we can assume that $A=(s,u]\times (t,v]$ for $(s,t), (u,v)\in\Pi_n$. In particular, we will exploit that in
this case both $t-s, v-u \ge n^{-1}$.\\[5pt]

{\it (2) Position of the block.} Secondly, we argue that it is sufficient to study two special types of blocks $A$ on the grid $\Pi_n$ only:
\begin{itemize}
	\item[(a)] $A = (s,u] \times (t,v]$ with $s< u \le t < v$, is termed a block ``above the diagonal'';
	\item[(b)] $A = (s,u] \times (s,u]$ with $s<u$, is termed a block ``on the diagonal''.
\end{itemize}
Indeed, every block $A \subset [0,1]^2$ which is not ``on the diagonal'' can be split into four blocks $A_1, A_2, A_3, A_4$ such that exactly 
one block, $A_1$ say, is of type (b) (``on the diagonal'') and the remaining three blocks, $A_2,A_3,A_4$, are either of type (a) (``above the 
diagonal'') or are of type (a) when being ``transposed''. By the latter we mean that the transpose of a block $B=(s,u]\times (t,v]$ is 
$B^T = (t,v]\times (s,u]$, i.e., the block $B$ is reflected along the diagonal. Due to the symmetry of the functional $h_2$ 
we have that the increment over a block coincides with the increment over the corresponding transposed block, i.e., 
$ U_n(h_2)(B) = U_n(h_2)(B^T)$, and hence we can assume w.l.o.g.\ that $A_2,A_3,A_4$ are all of type (a). 

Having made this partitioning of a block $A$ which is ``not on the diagonal'', observe that by definition of the $2$-dimensional increment, common corner points of neighboring 
blocks cancel each other and therefore $U_n(h_2)(A) = \sum_{i=1}^4 U_n(h_2)(A_i)$. Applying the Cauchy-Schwarz inequality yields
\begin{align}\label{Eq:TypeAIncrement}
	\E[ |\sqrt{n} U_n(h_2)(A)|^2 ] 
	\le 4 \sum_{i=1}^4 \E[ |\sqrt{n} U_n(h_2)(A_i)|^2 ].   
\end{align}
Thus, it remains to proof \eqref{E:MomentCondition0} for blocks $A$ of type (a) or type (b), because then - as an immediate consequence 
of \eqref{Eq:TypeAIncrement} - we conclude 
\begin{align*}
	\E[ |\sqrt{n} U_n(h_2)(A)|^2 ] \le 4C \ \sum_{i=1}^4 |A_i|^{3/2} \leq 4C |A|^{3/2}.
\end{align*}

{\it (3) Hölder continuity.} We use the Hölder continuity of $h_2$ (see \eqref{Eq:LipContinuityH2}) in combination with the $L^p$-$m$-property of the process to obtain continuity properties which are necessary for the moment bound. 
To this end, let $Y_1,\ldots,Y_4$ and  $Z_1,\ldots,Z_4$ be $M$-valued random elements with marginal distribution equal to $\p_X$ and an arbitrary joint distribution. Then
\begin{align}
	&| \E[ h_2(Y_1,Y_2) h_2(Y_3,Y_4)   ] - \E[h_2(Z_1,Z_2) h_2(Z_3,Z_4) ] | \nonumber \\
	=&| \E[ h_2(Y_1,Y_2)( h_2(Y_3,Y_4) - h_2(Z_3,Z_4) ) ] \nonumber\\
	&+ \E[ h_2(Z_3,Z_4)( h_2(Y_1,Y_2) - h_2(Z_1,Z_2) ) ] | \nonumber\\
	\leq& \E[|h_2(Y_1,Y_2)|^q]^{1/q}~ \E[ | h_2(Y_3,Y_4) - h_2(Z_3,Z_4) |^p  ]^{1/p}\nonumber\\
	&\quad + \E[|h_2(Z_3,Z_4)|^q]^{1/q}~ \E[ |h_2(Y_1,Y_2) - h_2(Z_1,Z_2) |^p  ]^{1/p}\nonumber\\
		&\le 3L~  \sup_{U,V \sim \p_X} \E[|h_2(U,V)|^q]^{1/q} \Big\{ \E [ (d(Y_3,Z_3) + d(Y_4,Z_4))^{\rho p}]^{1/p} \nonumber\\ 
	&\quad + \E[(d(Y_1,Z_1) + d(Y_2,Z_2))^{\rho p}]^{1/p} \Big\} \nonumber\\
	&\le 3L~\sup_{U,V \sim \p_X} \E[|h_2(U,V)|^q]^{1/q} \Big\{  \E [d(Y_3,Z_3)^{\rho p}]^{1/p} + \E[d(Y_4,Z_4)^{\rho p}]^{1/p} \nonumber\\ 
	&\quad\qquad\qquad\qquad\qquad + \E[d(Y_1,Z_1)^{\rho p}]^{1/p} + \E[ d(Y_2,Z_2)^{\rho p}]^{1/p} \Big\} \nonumber\\
	&= 3L \sup_{U,V \sim \p_X} \E[|h_2(U,V)|^q]^{1/q} \ \sum_{i =1}^4 \E[d(Y_i,Z_i)^{\rho p } ]^{1/p} \label{E:MomentCondition3},
\end{align}
where $p,q \geq 1$ are H\"older conjugate and where the second inequality is a consequence of \eqref{Eq:LipContinuityH2}; the last inequality is a consequence of the Minkowski inequality and the fact that $(a+b)^\rho \le a^\rho+b^\rho$ for $a,b\in\R_+$. 

Furthermore  
\begin{align*}
	\E[|h_2(U,V)|^q]^{1/q} 
	\leq \E[|h(U,V)|^q]^{1/q} + \theta + \E[|h_1(U)|^q]^{1/q} + \E[|h_1(V)|^q]^{1/q}
\end{align*}
by Minkowski's inequality and since - for $X'$ independent of $U$ -
\begin{align*}
	\E[|h_1(U)|^q]^{1/q} 
	&= \E\Big[ |\E[h(u,X')] - \theta|^q \Big|_{u=U} \Big]^{1/q} \\
	&\le \E[ |h(U,X')|^q ]^{1/q} + |\theta| 
	\leq \sup_{U,V\sim \p_X} \E[|h(U,V)|^q]^{1/q} + |\theta|,
\end{align*}
we have the upper bound 
\begin{align}
	\sup_{U,V \sim \p_X} \E[|h_2(U,V)|^q]^{1/q} 
	&\leq 3\left(\sup_{U,V \sim \p_X} \E[|h(U,V)|^q]^{1/q} + |\theta| \right) \nonumber \\
	&= 3(K(q)+|\theta|)\label{E:MomentCondition7}
\end{align}
with the definition from \eqref{Cond:Moment_h}.
Hence combining \eqref{E:MomentCondition3} and \eqref{E:MomentCondition7}, we obtain 
\begin{align}
	&\big|\E[ h_2(Y_1,Y_2) h_2(Y_3,Y_4)   ] - \E[h_2(Z_1,Z_2) h_2(Z_3,Z_4) ] \big|\nonumber\\
	\leq& 9L\big(K(q) + |\theta| \big) \sum_{i =1}^4 \E[d(Y_i,Z_i)^{\rho p}]^{1/p}. \label{E:MomentCondition8}
\end{align} 

{\it (4) Moment condition.} Finally, we verify the moment condition \eqref{E:MomentCondition0} for blocks $A=(s,u]\times(t,v]$, which have corner 
points on the grid $\Pi_n$ and which either lie ``above the diagonal'' (type (a)) or ``on the diagonal'' (type (b)).

{\it (a) A block ``above the diagonal''}. In this case $s< u \le t< v$ and therefore
\begin{align*}
	U_n(h_2)(A) = \frac{1}{n(n-1)} \sum_{i=\lf ns\rf + 1}^{\lf nu \rf} \sum_{j=\lf nt \rf + 1}^{\lf nv \rf} h_2(X_i,X_j),
\end{align*}
which yields
\begin{align} 
	\begin{split}\label{E:MomentCondition1}
	&n^2 (n-1)^2 \E[ | U_n(h_2)(A) |^2 ] \\
	 &=  \ \sum_{i,k = \lf ns \rf +1}^{ \lf nu\rf}  \sum_{j,\ell = \lf nt \rf+1}^{\lf nv\rf} | \E[ h_2(X_i,X_j)h_2(X_k, X_\ell) ]|. 
	 \end{split}
\end{align}
We focus w.l.o.g.\ on the subcase where $i\le k$ because of the symmetry of \eqref{E:MomentCondition1} with respect
to these indices. Moreover, we have that 
both $j,\ell > k$ because $u\le t$. Consequently, we only have to study two cases in order to derive the moment condition for \eqref{E:MomentCondition1}:
(i) $i\le k < j \le \ell$ and (ii) $i\le k < \ell \le j$. We will consider (i) in detail and as a by-product we will see that (ii) also works in the same spirit.

{\it Subcase (i)}: Let us make the following abbreviations
$$
	\alpha = k-i \ge 0, \qquad \beta = j-k > 0, \qquad \gamma = \ell-j \ge 0,
$$
so that we can utilize the stationarity of the process and perform a simple substitution as follows
\begin{align}
	&\sum_{i,k = \lf ns \rf +1}^{ \lf nu\rf}  \sum_{j,\ell = \lf nt \rf+1}^{\lf nv\rf} \1{ i\le k < j \le \ell } 
	| \E[ h_2(X_i,X_j)h_2(X_k, X_\ell) ]| \nonumber \\
	\begin{split}
	&=  \sum_{i = \lf ns \rf +1}^{\lf nu \rf} \sum_{\alpha=\lf ns\rf +1-i}^{\lf nu \rf-i} \sum_{\beta=\lf nt\rf +1-i-\alpha}^{\lf nv \rf -i-\alpha} 
	\sum_{\gamma=\lf nt\rf+1-i-\alpha-\beta}^{\lf nv \rf-i-\alpha-\beta} \1{ \alpha\ge 0, \beta> 0, \gamma\ge 0} \\
	&\qquad\qquad\qquad\qquad \cdot | \E[ h_2(X_0,X_{\alpha+\beta} )h_2(X_\alpha, X_{\alpha+\beta+\gamma} ) ]| .\label{E:MomentCondition2}  
	\end{split}
\end{align}
Next, we distinguish another three subcases depending on which of the three indices $\alpha,\beta,\gamma$ is the largest. 
This will allow us to utilize the $L^p$-$m$-approximation property of the process $(X_t)_t$ in order to verify a decay of the expectations which is sufficiently fast. 
To that end, recall that $X_t = f(\epsilon_t, \epsilon_{t-1},\ldots)$ and 
$X_t^{(m)} = f(\epsilon_t,\ldots, \epsilon_{t-m+1},$ $\epsilon_{t-m}^{(t)}, \epsilon_{t-m-1}^{(t)},\ldots  )$ as well as 
$\E[h_2(x,X')] = 0$ for all $x\in\R^d$ and $X' \sim \p_X$. Therefore, 
\begin{itemize}
	\item [($\alpha$)] if $\alpha$ is among the largest indices, then
	\begin{align*}
		& \E[ h_2(X_0,X_{\alpha+\beta}^{(\alpha)} ) h_2(X_\alpha^{(\alpha)}, X_{\alpha+\beta+\gamma}^{(\alpha)} ) ] \\
		&= \E\Big[ \E[ h_2(X_0,x)]\big|_{x=X_{\alpha+\beta}^{(\alpha)}} h_2(X_\alpha^{(\alpha)}, X_{\alpha+\beta+\gamma}^{(\alpha)} )  \Big] = 0
	\end{align*}
	because $X_0$ is independent of the other three covariables.
	\item [($\beta$)] if $\beta$ is among the largest indices, then
	\begin{align*}
		& \1{\alpha\le\gamma\le\beta} \E[ h_2(X_0,X_{\alpha+\beta} )h_2(X_\alpha, X_{\alpha+\beta+\gamma}^{(\gamma)} ) ]\\ 
		&= \1{\alpha\le\gamma\le\beta} \E\Big[ h_2(X_0,X_{\alpha+\beta} ) \E[h_2(x, X_{\alpha+\beta+\gamma}^{(\gamma)} )]\big|_{x=X_{\alpha}} \Big] = 0
	\end{align*}
	because $X_{\alpha+\beta+\gamma}^{(\gamma)}$ is independent of the other three covariables. Similarly
	\begin{align*}
		& \1{\gamma\le\alpha\le\beta} \E[ h_2(X_0,X_{\alpha+\beta}^{(\alpha)} )h_2(X_\alpha^{(\alpha)}, X_{\alpha+\beta+\gamma}^{(\alpha)} ) ]\\
		&= \1{\gamma\le\alpha\le\beta} \E\Big[\E[h_2(X_0,x)]\big|_{x=X_{\alpha+\beta}^{(\alpha)}} 
		h_2(X_\alpha^{(\alpha)}, X_{\alpha+\beta+\gamma}^{(\alpha)} )  \Big]=0
	\end{align*}
	because in this case $X_0$ is independent.
	\item [($\gamma$)] if $\gamma$ is among the largest indices, then
	\begin{align*} 
		 & \E[ h_2(X_0,X_{\alpha+\beta} )h_2(X_\alpha, X_{\alpha+\beta+\gamma}^{(\gamma)} ) ]\\ 
		 &= \E\Big[ h_2(X_0,X_{\alpha+\beta} ) \E[h_2(x, X_{\alpha+\beta+\gamma}^{(\gamma)} )]\big|_{x=X_{\alpha}} \Big] = 0
	\end{align*}
	because $X_{\alpha+\beta+\gamma}^{(\gamma)}$ is independent of the other covariables.
\end{itemize}
Next, we can return to \eqref{E:MomentCondition2}: In each subcase, we can subtract the appropriate expectation from the expectation given in the last factor in \eqref{E:MomentCondition2} and then use the Lipschitz-property of the coupling as detailed in \eqref{E:MomentCondition8}. Going through the single subcases, we 
immediately see
\begin{align}
	\eqref{E:MomentCondition2} &\le \sum_{i = \lf ns \rf +1}^{\lf nu \rf} \sum_{\alpha=\lf ns\rf +1-i}^{\lf nu \rf-i} \sum_{\beta=\lf nt\rf +1-i-\alpha}^{\lf nv \rf -i-\alpha} 
	\sum_{\gamma=\lf nt\rf+1-i-\alpha-\beta}^{\lf nv \rf-i-\alpha-\beta} \1{ \alpha\ge 0, \beta> 0, \gamma\ge 0} \nonumber \\
	&\quad \cdot \Big\{ \big|\E[h_2(X_0,X_{\alpha+\beta})h_2(X_{\alpha}X_{\alpha+\beta+\gamma})] 
		- \E[h_2(X_0,X_{\alpha+\beta}^{(\alpha)})h_2(X_{\alpha}^{(\alpha)},X_{\alpha+\beta+\gamma}^{(\alpha)})]\big| \nonumber\\
		&\quad\qquad\times \1{ \beta \vee \gamma \leq \alpha} \nonumber\\
	&\quad + \big|\E[h_2(X_0,X_{\alpha+\beta})h_2(X_{\alpha}X_{\alpha+\beta+\gamma})] 
		- \E[h_2(X_0,X_{\alpha+\beta})h_2(X_{\alpha},X_{\alpha+\beta+\gamma}^{(\gamma)})]\big| \nonumber\\
		&\quad\qquad\times \1{ \alpha \leq \gamma \leq \beta} \nonumber\\
	&\quad + \big|\E[h_2(X_0,X_{\alpha+\beta})h_2(X_{\alpha}X_{\alpha+\beta+\gamma})] 
		- \E[h_2(X_0,X_{\alpha+\beta}^{(\alpha)})h_2(X_{\alpha}^{(\alpha)},X_{\alpha+\beta+\gamma}^{(\alpha)})]\big| \nonumber\\
		&\qquad\times \1{ \gamma \leq \alpha \leq \beta} \nonumber\\
	&\quad + \big|\E[h_2(X_0,X_{\alpha+\beta})h_2(X_{\alpha}X_{\alpha+\beta+\gamma})] 
		- \E[h_2(X_0,X_{\alpha+\beta})h_2(X_{\alpha},X_{\alpha+\beta+\gamma}^{(\gamma)})]\big| \nonumber\\
		&\qquad\times \1{ \alpha \vee \beta \leq \gamma} 
	\Big\}\nonumber\\
	&\lesssim \sum_{i = \lf ns \rf +1}^{\lf nu \rf} \sum_{\alpha=\lf ns\rf +1-i}^{\lf nu \rf-i} \sum_{\beta=\lf nt\rf +1-i-\alpha}^{\lf nv \rf -i-\alpha} 
	\sum_{\gamma=\lf nt\rf+1-i-\alpha-\beta}^{\lf nv \rf-i-\alpha-\beta} \1{ \alpha\ge 0, \beta> 0, \gamma\ge 0} \nonumber \\
	&\quad \cdot 9L( K(q) + \theta) \ \Big\{ 
		\E[d(X_0,X_0^{(\alpha)})^{\rho p}]^{1/p} \1{ \beta \vee \gamma \leq \alpha} \label{E:MomentCondition4}\\
	\begin{split}\label{E:MomentCondition5} 
	&\quad + \E[d(X_0,X_0^{(\gamma)})^{\rho p} ]^{1/p} \1{ \alpha \leq \gamma \leq \beta} \\
	&\quad + \E[d(X_0,X_0^{(\alpha)})^{\rho p}]^{1/p} \1{ \gamma \leq \alpha \leq \beta}
	\end{split} \\ 
	&\quad + \E[d(X_0,X_0^{(\gamma)})^{\rho p}]^{1/p} \1{ \alpha \vee \beta \leq \gamma}\Big\}. \label{E:MomentCondition6}
\end{align}
Here the upper bounds in \eqref{E:MomentCondition4} (resp.,  \eqref{E:MomentCondition5} and \eqref{E:MomentCondition6}) correspond to the subcase $(\alpha)$ 
(resp., $(\beta)$ and $(\gamma)$).
Next, summing up the last four sums, yields the following upper bound (given for clarity with four terms)
\begin{align*}
	\eqref{E:MomentCondition2}	&\lesssim \sum_{i = \lf ns \rf +1}^{\lf nu \rf } \sum_{\beta=\lf nt \rf+1}^{\lf nv \rf} \sum_{\alpha\ge 0} \alpha ~ \E[d(X_0,X_0^{(\alpha)})^{\rho p}]^{1/p} \\ 
	&\quad+ \sum_{i = \lf ns \rf +1}^{\lf nu \rf} \sum_{\beta=\lf nt \rf +1}^{\lf nv \rf} \sum_{\gamma \geq 0} \gamma ~\E[d(X_0,X_0^{(\gamma)})^{\rho p}]^{1/p} \\
	&\quad+ \sum_{i = \lf ns \rf +1}^{\lf nu\rf} \sum_{\beta=\lf nt \rf +1}^{\lf nv \rf} \sum_{\alpha\ge 0} \alpha ~ \E[d(X_0,X_0^{(\alpha)})^{\rho p}]^{1/p} \\ 
	&\quad+ \sum_{i = \lf ns \rf+1}^{\lf nu\rf} \sum_{\beta=\lf nt \rf+1}^{\lf nv \rf} \sum_{\gamma\ge 0} \gamma ~ \E[d(X_0,X_0^{(\gamma)})^{\rho p}]^{1/p} \\
	&\lesssim n^2 (u-s) (v-t)
\end{align*}
by the approximating property \eqref{Cond:Approximation} in Assumption~\ref{AssumptionFCLT}.

{\it Subcase (ii)}: It is straightforward to see that a similar calculation is valid if $i\le j < \ell \le k$.

Hence, we obtain the following upper bound for \eqref{E:MomentCondition1}
\begin{align*}
	n^2 (n-1)^2 \E[ |U_n(h_2)(A) | ^2 ] \lesssim n^2 (u-s)(v-t) = n^2 |A|.
\end{align*}
Consequently, there is a constant $C\in\R_+$ such that for a block $A$ ``above the diagonal'' and with corner points on the grid $\Pi_n$ we have
\begin{align}\label{Eq:UpperBoundAboveDiagBlock}
	\E[ |\sqrt{n} U_n(h_2)(A) |^2 ] \le C \frac{ |A| }{n} \le C |A|^{3/2},
\end{align}
where the last step follows because $\sqrt{|A|} = \sqrt{ (u-s) (v-t) } \ge \sqrt{n^{-1} n^{-1}} = n^{-1}$.\\[5pt]

{\it (b) A block ``on the diagonal''.} In this case $s=t < u = v$ and therefore (cf. \eqref{E:MomentCondition1}) 
\begin{align}
\begin{split}\label{E:MomentConditionDiagonalBlock}
	&n^2 (n-1)^2 \E[ | U_n(h_2)(A) |^2 ] \\
	=&  \ \sum_{i,k,j,\ell = \lf ns \rf +1}^{ \lf nu\rf} | \E[ h_2(X_i,X_j)h_2(X_k, X_\ell) ]|. 
\end{split}
\end{align}
Once more, taking into account the symmetry of the problem, we only need to consider the cases where $i\le k$. Furthermore, it is sufficient to 
consider only the following three instances
\begin{align*}
	 &(i) \quad i < j \le k < \ell, \qquad (ii) \quad i \le k \le j \le \ell, \qquad (iii) \quad i \le k < \ell \le j.
\end{align*}
For example, the first case also encompasses the cases $j<i\leq k < \ell$ and $i<j\leq \ell < k$ and $j<i\leq \ell <k$ because 
\begin{align*}
	\E[h_2(X_i,X_j)h_2(X_k,X_\ell)] 
	=& \E[h_2(X_j,X_i)h_2(X_k,X_\ell)]\\
	=& \E[h_2(X_i,X_j)h_2(X_\ell,X_k)]	= \E[h_2(X_j,X_i)h_2(X_\ell,X_k)].
\end{align*}
Now the cases (i), (ii) and (iii) essentially work in the same spirit and again rely heavily on the property of $h_2$ being degenerate as well as on the knowledge 
of the relative positions of the indices $i,j,k,\ell$.

{\it Case (i)}: Using the stationarity of the process, we obtain
\begin{align}
	&\sum_{\lf ns \rf +1\le i < j \le k < \ell\le \lf nu\rf} | \E[ h_2(X_i,X_j) h_2(X_k,X_\ell) ]  | \nonumber \\
	&= \sum_{\lf ns \rf +1\le i < j \le k < \ell\le \lf nu\rf} | \E[ h_2(X_0,X_{j-i}) h_2(X_{k-i},X_{\ell-i}) ]  | \nonumber \\
	&\le n(u-s) \sum_{\lf ns \rf +1\le o,r,w \le \lf nu\rf} | \E[ h_2(X_0,X_{o}) h_2(X_{o+r},X_{o+r+w}) ]  |. \label{E:VarDegenerateU1}
\end{align}
We split the sum in \eqref{E:VarDegenerateU1} in three subsums depending on the relative positions of the three indices: $(\alpha)$ $o,r \le w$, $(\beta)$ $r,w \le o$ and $(\gamma)$ $o,w\le r$. 
Depending on the subcase, we proceed similar as in the setting for a block ``above the diagonal'' and use a suitable coupling in combination with
\eqref{E:MomentCondition8} to derive a fast enough decay of the expectation.

{\it Subcase $(\alpha)$}:
\begin{align*}
		&\sum_{\lf ns \rf +1 \le o,r \le w \le \lf nu\rf} | \E[ h_2(X_0,X_{o}) h_2(X_{o+r},X_{o+r+w}) ]  | \\
		&= \sum_{\lf ns \rf +1 \le o,r \le w \le \lf nu\rf} | \E[ h_2(X_0,X_{o}) h_2(X_{o+r},X_{o+r+w})] - \E[h_2(X_0,X_{o}) h_2(X_{o+r},X^{(w)}_{o+r+w}) ]  | \\
		&\lesssim \sum_{\lf ns \rf +1 \le o,r \le w \le \lf nu\rf} \E[ d(X_0,X_0^{(w)})^{\rho p}]^{1/p} \\
		&\lesssim n(u-s) \sum_{w=1}^\infty w~  \E[ d(X_0,X_0^{(w)})^{\rho p}]^{1/p}.
\end{align*}	

{\it Subcase $(\beta)$:}
\begin{align*}
		&\sum_{\lf ns \rf +1\le r,w \le o \le \lf nu\rf} | \E[ h_2(X_0,X_{o}) h_2(X_{o+r},X_{o+r+w})]  | \\
		&\le \sum_{\lf ns \rf +1\le r \le w \le o \le \lf nu\rf} | \E[ h_2(X_0,X_{o}) h_2(X_{o+r},X_{o+r+w})]  | \\
		&\quad \qquad\qquad + \sum_{\lf ns \rf +1\le w \le r \le o \le \lf nu\rf} | \E[ h_2(X_0,X_{o}) h_2(X_{o+r},X_{o+r+w})]  | \\
		&\le \sum_{\lf ns \rf +1\le r \le w \le o \le \lf nu\rf} 
		| \E[ h_2(X_0,X_{o}) h_2(X_{o+r},X_{o+r+w})]\\
		&\quad \qquad\qquad - \E[ h_2(X_0,X_{o}) h_2(X_{o+r},X_{o+r+w}^{(w)})]  | \\
		&\quad \qquad\qquad + \sum_{\lf ns \rf +1\le w \le r \le o \le \lf nu\rf} 
		| \E[ h_2(X_0,X_{o}) h_2(X_{o+r},X_{o+r+w})] \\ 
		&\quad \qquad\qquad - \E[h_2(X_0,X^{(o)}_{o}) h_2(X^{(r)}_{o+r},X^{(r)}_{o+r+w})]  | \\
		&\lesssim n(u-s) \sum_{w=1}^\infty w~ \E[ d(X_0,X_0^{(w)})^{\rho p}]^{1/p}
		+ n(u-s) \sum_{o=1}^\infty o~ \E[ d(X_0,X_0^{(o)})^{\rho p}]^{1/p}\\
		&\quad \qquad\qquad+ n(u-s) \sum_{r=1}^\infty r~ \E[ d(X_0,X_0^{(r)})^{\rho p}]^{1/p}.
\end{align*}	

{\it Subcase $(\gamma)$:}
\begin{align*}
		&\sum_{\lf ns \rf +1\le o,w \le r \le \lf nu\rf} | \E[ h_2(X_0,X_{o}) h_2(X_{o+r},X_{o+r+w})]  | \\
		&\le \sum_{\lf ns \rf +1\le o \le w \le r \lf nu\rf} | \E[ h_2(X_0,X_{o}) h_2(X_{o+r},X_{o+r+w})] \\
		&\quad \qquad\qquad - \E[ h_2(X_0,X_{o}) h_2(X_{o+r},X_{o+r+w}^{(w)})]  | \\
		& + \sum_{1\le w\le v \le u \le n} | \E[ h_2(X_0,X_{o}) h_2(X_{o+r},X_{o+r+w}) ] \\
		&\quad \qquad\qquad- \E[h_2(X_0,X^{(o)}_{o}) h_2(X^{(r)}_{o+r},X^{(r)}_{o+r+w})]  | \\
		&\lesssim n(u-s) \sum_{w=1}^\infty w~ \E[ d(X_0,X_0^{(w)})^{\rho p}]^{1/p}
		+ n(u-s) \sum_{o=1}^\infty o~ \E[ d(X_0,X_0^{(o)})^{\rho p}]^{1/p}\\
		&\quad \qquad\qquad + n(u-s) \sum_{r=1}^\infty r~  \E[ d(X_0,X_0^{(r)})^{\rho p}]^{1/p}.
\end{align*}
Thus, collecting the three subcases and returning to \eqref{E:VarDegenerateU1}, we see once more with the approximating property from Equation \eqref{Cond:Approximation} that
\begin{align*}
	\sum_{\lf ns \rf +1\le i < j \le k < \ell \le \lf nu\rf} | \E[ h_2(X_i,X_j) h_2(X_k,X_{\ell}) ]  | 
	\lesssim n^2(u-s)^2.
\end{align*}
One finds with similar ideas that in \\
{\it Case (ii)}:
\begin{align*}
	&\sum_{1\le i \le k \le j \le \ell\le n} | \E[ h_2(X_i,X_j) h_2(X_k,X_\ell) ]  | 
	\lesssim n^2(u-s)^2;
\end{align*}
{\it Case (iii)}:
\begin{align*}
	&\sum_{1\le i  \le k < \ell \le j \le n} | \E[ h_2(X_i,X_j) h_2(X_k,X_\ell) ]  | 
	\lesssim n^2(u-s)^2.
\end{align*}
Consequently, there is a constant $C\in\R_+$ such that for a block $A$ ``on the diagonal'' and with corner points on the grid $\Gamma_n$  we have
\begin{align}\label{Eq:UpperBoundDiagBlock}
	\E[ |\sqrt{n} U_n(h_2)(A) |^2 ] \le C \frac{ |A| }{n} \le C |A|^{3/2}.
\end{align}
This completes the proof.
\end{proof}

\begin{proof}[Proof of Corollary~\ref{T:VarDegenerateU}]
Observe that the $2$-dimensional increment of $U_n(h_2)$ over $A=(0,1]\times (0,1]$ is simply
\begin{align*}
	U_n(h_2)(A) = U_n(h_2)(1,1) = U_n(h_2) = \frac{2}{n(n-1)} \sum_{1\leq i<j\leq n} h_2(X_i,X_j).
\end{align*}
Using the upper bound \eqref{Eq:UpperBoundDiagBlock} which is derived in the proof Theorem \ref{T:MomentCondition} we conclude 
\begin{align*}
	& \E\left[ \left|\sqrt{n} \ U_n(h_2)(A)  \right|^2 \right]
	=  \ \E\left[ \left|\frac{2}{\sqrt{n}(n-1)} \ \sum_{1\leq i<j\leq n} h_2(X_i,X_j)  \right|^2 \right]\\
	&= \frac{4}{n(n-1)^2} \sum_{1\leq i<j\leq n} \sum_{1\leq k<l\leq n} \E[ h_2(X_i,X_j)h_2(X_k,X_\ell)  ] \le \frac{C}{n}
\end{align*} 
for a certain $C\in\R_+$.
\end{proof}

\begin{cor}[Variance of $U_n(h)$]\label{C:VarUh}
Granted Assumption~\ref{AssumptionFCLT} (p,q,$\rho$) is satisfied. Then there is a universal constant $C'\in\R_+$, such that
$$
		\E\Big[ \Big( \frac{2}{n(n-1)} \sum_{1\le i<j\le n } h(X_i,X_j) - \theta \Big)^2 \Big] \le C' n^{-1}.
$$
\end{cor}
\begin{proof}
\begin{align}
	&\E\Big[ \Big( \frac{2}{n(n-1)} \sum_{1\le i<j\le n } h(X_i,X_j) - \theta \Big)^2 \Big] \nonumber \\
	&\le 3 \ \Big( \frac{2}{n(n-1)} \Big)^2 ~ \Big\{ \E\Big[ \Big( \sum_{1\le i<j\le n } h_2(X_i,X_j) \Big)^2 \Big] 
	+ \frac{(n-1)^2}{2} \E\Big[ \Big( \sum_{1\le i \le n } h_1(X_i) \Big)^2 \Big]  \Big\}. \label{E:VarUh0}
\end{align}
The first term inside the curly parentheses of \eqref{E:VarUh0} is of order $n^2$ by Corollary~\ref{T:VarDegenerateU}.
The second term inside the curly parentheses is of order $n^2\cdot n$. Indeed, let $X'$ and $X''$ be two independent random variables with law $\p_{X_0}$, 
then by Assumption~\ref{AssumptionFCLT} (p,q,$\rho$)
\begin{align*}
	&\sum_{1\le i \le j\le n} \E[ h_1(X_i)h_1(X_j) ] = \sum_{1\le i,j\le n} \E[ (h(X_0,X')-\theta)(h(X_{j-i},X'')-\theta) ] \\
	&= \sum_{1\le i,j\le n} \E[ (h(X_0,X')-\theta)(h(X_{j-i},X'')- h(X_{j-i}^{(j-i)}, X'') ) ] \\
	&\le \| h(X_0, X') \|_q ~  \sum_{1\le i,j\le n} L \| d(X_{j-i},X_{j-i}^{(j-i)})^\rho \|_p \lesssim n.
\end{align*}
This completes the proof.
\end{proof}

\subsection{Auxiliary results}

The bottleneck distance between persistence diagrams $D_1$, $D_2$ is given by
$$
	W_B( D_1, D_2 ) = W_\infty(D_1,D_2) = \inf_{\gamma} \sup_{p\in D_1} \|p - \gamma(p) \|_\infty
$$
where $\gamma$ is the set of bijections between the multi-sets $D_1$ and $D_2$ (a point with multiplicity $m>1$ is considered as $m$ disjoint copies) and $\|p-q\|_\infty=\max\{ |x_p-x_q|, |y_p-y_q| \}$).

\begin{proposition}\label{P:Stability}[Chazal et al.\ (2014) \cite{chazal2014persistence}, Cohen-Steiner et al.\ 2007 \cite{cohen2007stability}]
Let $X,Y$ be two totally bounded metric spaces. Let $\pd(X)$, resp., $\pd(Y)$ be the persistence diagrams obtained from $X$, resp., $Y$ using the \v Cech or Vietoris-Rips filtration. Then
\begin{align*}
	W_B( \pd(X), \pd(Y) ) \le 2 d_{GH}(X,Y)
\end{align*}
where $d_{GH}$ is the Gromov-Hausdorff distance.
\end{proposition}

\begin{lemma}\label{L:PDandHD}
Let $X,Y$ be two finite subsets in a metric space $(M,d)$. Let $\pd(X)$, resp., $\pd(Y)$ be the persistence diagrams obtained from $X$, resp., $Y$ using the \v Cech or Vietoris-Rips filtration.
Then
\begin{align*}
	W_B( \pd(X), \pd(Y) ) \le 2 d_{H}(X,Y)
\end{align*}
where
$$
	d_H(X,Y) = \max\Big\{ \sup_{x\in X} \inf_{y\in Y} d(x,y), \sup_{y\in Y} \inf_{x\in X}  d(x,y)	\Big\}
$$
is the Hausdorff distance between $X$ and $Y$.
\end{lemma}
\begin{proof}
We apply Proposition~\ref{P:Stability} and use $d_{GH}(X,Y)\le d_H(X,Y)$ in the present case because the identity on $M$ is an isometric embedding.
\end{proof}

\section*{Acknowledgments}
Johannes Krebs and Daniel Rademacher are grateful for the financial support of the German Research Foundation (DFG), Grant Number KR-4977/2-1.


\end{document}